\documentclass[a4paper]{amsart}
\usepackage{amssymb, enumitem}
\usepackage[all]{xy}
\usepackage{hyperref, aliascnt}
\usepackage{bbold, mathtools}

\DeclareMathOperator{\Cu}{Cu}
\DeclareMathOperator{\Lsc}{Lsc}
\DeclareMathOperator{\LAff}{LAff}
\DeclareMathOperator{\Aff}{Aff}
\DeclareMathOperator{\QT}{QT}
\DeclareMathOperator{\rk}{rk}
\DeclareMathOperator{\tr}{tr}
\newcommand{\txtSoft}{\mathrm{soft}}
\newcommand{\txtLsc}{\mathrm{lsc}}
\newcommand{\charFct}{\mathbb{1}}
\newcommand{\andSep}{\,\,\,\text{ and }\,\,\,}
\newcommand{\CuSgp}{$\CatCu$-sem\-i\-group}
\newcommand{\CuMor}{$\CatCu$-mor\-phism}
\newcommand{\NNbar}{\overline{\NN}}
\newcommand{\dimnuc}{\dim_{\text{nuc}}}
\newcommand{\pom}{positively ordered monoid}
\newcommand{\AW}{$AW^*$}
\newcommand{\ZZ}{{\mathbb{Z}}}
\newcommand{\NN}{{\mathbb{N}}}
\newcommand{\KK}{{\mathbb{K}}}
\newcommand{\CC}{{\mathbb{C}}}
\newcommand{\RR}{{\mathbb{R}}}
\newcommand{\id}{{\mathrm{id}}}
\newcommand{\dist}{{\mathrm{dist}}}
\newcommand{\ca}{$C^*$-al\-ge\-bra}
\newcommand{\stHom}{${}^*$-homomorphism}
\newcommand{\CatCu}{\ensuremath{\mathrm{Cu}}}
\newcommand{\axiomO}[1]{(O#1)}
\newcommand{\Bdd}{\mathcal{B}}

\newtheorem{lma}{Lemma}[section]

\numberwithin{equation}{lma}

\newaliascnt{thmCt}{lma}
\newtheorem{thm}[thmCt]{Theorem}
\aliascntresetthe{thmCt}

\newaliascnt{corCt}{lma}
\newtheorem{cor}[corCt]{Corollary}
\aliascntresetthe{corCt}

\newaliascnt{prpCt}{lma}
\newtheorem{prp}[prpCt]{Proposition}
\aliascntresetthe{prpCt}

\newtheorem*{thm*}{Theorem}
\newtheorem*{cor*}{Corollary}
\newtheorem*{prop*}{Proposition}

\theoremstyle{definition}

\newaliascnt{pgrCt}{lma}
\newtheorem{pgr}[pgrCt]{}
\aliascntresetthe{pgrCt}

\newaliascnt{dfnCt}{lma}
\newtheorem{dfn}[dfnCt]{Definition}
\aliascntresetthe{dfnCt}

\newaliascnt{rmkCt}{lma}
\newtheorem{rmk}[rmkCt]{Remark}
\aliascntresetthe{rmkCt}

\newaliascnt{rmksCt}{lma}
\newtheorem{rmks}[rmksCt]{Remarks}
\aliascntresetthe{rmksCt}

\newaliascnt{exaCt}{lma}
\newtheorem{exa}[exaCt]{Example}
\aliascntresetthe{exaCt}

\newaliascnt{qstCt}{lma}
\newtheorem{qst}[qstCt]{Question}
\aliascntresetthe{qstCt}

\newaliascnt{cnjCt}{lma}
\newtheorem{cnj}[cnjCt]{Conjecture}
\aliascntresetthe{cnjCt}

\title{Ranks of operators in simple \texorpdfstring{$C^*$-algebras}{C*-algebras} with stable rank one}
\date{\today}

\author{Hannes Thiel}
\address{Hannes Thiel
Mathematisches Institut, Fachbereich Mathematik und Informatik der
Universit\"at M\"unster, Einsteinstrasse 62, 48149 M\"unster, Germany.}
\email{hannes.thiel@uni-muenster.de}
\urladdr{www.math.uni-muenster.de/u/hannes.thiel/}

\thanks{The author was partially supported by the Deutsche Forschungsgemeinschaft (SFB 878 Groups, Geometry \& Actions).}


\subjclass[2010]%
{Primary
46L05; 
Secondary
06B35, 
06F05, 
19K14, 
46L35, 
46L80. 
}

\begin{document}

\begin{abstract}
Let $A$ be a simple \ca{} with stable rank one.
We show that every strictly positive, lower semicontinuous, affine function on the simplex of normalized quasitraces of $A$ is realized as the rank of an operator in the stabilization of $A$.

Assuming moreover that $A$ has locally finite nuclear dimension, we deduce that $A$ is $\mathcal{Z}$-stable if and only if it has strict comparison of positive elements.
In particular, the Toms-Winter conjecture holds for simple, approximately subhomogeneous \ca{s} with stable rank one.
\end{abstract}

\maketitle

\section{Introduction}

The rank of a matrix is one of the most fundamental notions in linear algebra.
In this context, the following two standard facts concerning matrix ranks are repeatedly
used:
\begin{itemize}
\item
\emph{Comparison}:
We have $\rk(x)\leq\rk(y)$ if and only if $x=rys$ for some $r,s$.
\item
\emph{Range}:
$n\times n$-matrices realize the ranks $0,1,2,\ldots,n$.
\end{itemize}

The rank of a projection can be computed as its trace.
Similarly, the rank of a projection in the algebra $\Bdd(H)$ of bounded operators on a separable, infinite-dimensional Hilbert space $H$ is computed by the canonical (unbounded) trace $\tr$.
If two projections $p$ and $q$ satisfy $p=rqs$ for some $r,s$, then one says that $p$ is \emph{Murray-von Neumann subequivalent} to $q$, denoted $p\precsim q$.
The following facts hold for projections in $\Bdd(H)$:
\begin{itemize}
\item
\emph{Comparison}:
We have $\tr(p)\leq\tr(q)$ if and only if $p\precsim q$.
\item
\emph{Range}:
Projections in $\Bdd(H)$ realize the ranks $0,1,2,\ldots,\infty$.
\end{itemize}

There is no bounded trace on $\Bdd(H)$, but Murray and von Neumann discovered the class of $\mathrm{II}_1$ factors, which are simple, infinite-dimensional von Neumann algebras such that the unit is a finite projection.
Every $\mathrm{II}_1$ factor has a unique bounded trace (normalized at the unit).
Murray and von Neumann proved the following fundamental facts about projections in a $\mathrm{II}_1$ factor $M$:
\begin{enumerate}
\item[(C)]
\emph{Comparison}:
We have $\tr(p)\leq\tr(q)$ if and only if $p\precsim q$.
\item[(R)]
\emph{Range}:
For every $t\in[0,1]$ there is a projection $p\in M$ with $\tr(p)=t$.
\end{enumerate}

It is natural to ask whether these two properties have analogues for simple, tracial \ca{s}. 
The analogue of (C) is called \emph{strict comparison} (see \autoref{rmk:TW} for the definition), and it is known \emph{not} to hold automatically, unlike for $\mathrm{II}_1$ factors.

On the other hand, it is not known if the analogue of (R) is automatic for simple \ca{s}.
The main result of this paper is that it is automatic for simple \ca{s} with stable rank one.

Recall that a unital \ca{} is said to have \emph{stable rank one} if its invertible elements form a dense subset.
The stable rank is a noncommutative dimension theory that was introduced to study nonstable $K$-theory.
Just as topological spaces are more tractable if they have low dimension, \ca{s} with minimal stable rank (that is, stable rank one) are accessible to techniques that do not apply in general.

Let us describe the \ca{ic} analogue of (R).
While $\mathrm{II}_1$ factors have a unique normalized trace, this is no longer the case for simple \ca{s}.
This means that there is more than one way to `measure' the rank of a projection.
Given a general unital, simple \ca{} $A$, the normalized traces on $A$ form a Choquet simplex $T_1(A)$.
The \emph{rank} of a projection $p$ in $A$ is defined as the function $\rk(p)\colon T_1(A)\to [0,1]$, given by $\rk(p)(\tau):=\tau(p)$.
The question is then which functions $T_1(A)\to[0,1]$ are realized as the rank of a projection in $A$.
Given a $\mathrm{II}_1$ factor $M$, the space $T_1(M)$ is a singleton.
Therefore, property (R) says that every function $T_1(M)\to[0,1]$ is realized by a projection in $M$.

To formulate the accurate \ca{ic} analogue of (R), we have to apply two changes.
First, we have to replace traces on $A$ by $2$-quasitraces (see \autoref{pgr:qt} for definitions).
This is a minor change that can be ignored in many important cases.
For example, for exact \ca{s} there is no distinction between traces and $2$-quasitraces.
Second, we need to replace projections in $A$ by positive elements in the stabilization $A\otimes\KK$.
This change is more fundamental and cannot be ignored in applications.
The reason is that many interesting \ca{s} have no nontrivial projections.
In this case it is necessary to consider the rank of more general elements to obtain the correct analogue of (R).

As for traces, the normalized $2$-quasitraces form a Choquet simplex $\QT_1(A)$.
Given a positive element $x$ in $A\otimes\KK$, the \emph{rank} of $x$ at $\tau\in\QT_1(A)$ is defined as
\[
d_\tau(x) = \lim_{n\to\infty} \tau( x^{1/n} ).
\]
We call the resulting map $\rk(x)\colon \QT_1(A)\to[0,\infty]$, given by $\rk(x)(\tau):=d_\tau(x)$, the \emph{rank} of $x$.
The function $\rk(x)$ is lower semicontinuous and affine.
If $x\neq 0$, then $\rk(x)$ is also strictly positive, and we write $\rk(x)\in\LAff(\QT_1(A))_{++}$;
see \autoref{pgr:LAff}.
Thus, the precise question is:

\begin{qst}
\label{qst:range}
Let $A$ be a separable, unital, simple, non-elementary, stably finite \ca{}, 
and let $f\in\LAff(\QT_1(A))_{++}$.
Is there $x\in(A\otimes\KK)_+$ with $\rk(x)=f$?
\end{qst}

This question was first explicitly posed by N.~Brown.
Affirmative answers have been obtained in the following cases:
\begin{enumerate}
\item
under the additional assumption that $A$ tensorially absorbs the Jiang-Su algebra $\mathcal{Z}$, by Elliott-Robert-Santiago \cite[Corollary~6.8]{EllRobSan11Cone}, extending earlier work by Brown-Perera-Toms that covered the tensorially $\mathcal{Z}$-absorbing, exact case \cite[Theorem~5.5]{BroPerTom08CuElliottConj};
\item
under the assumption that $A$ is exact, has strict comparison of positive elements, and such that $\QT_1(A)$ is a Bauer simplex whose extreme boundary has finite covering dimension by Dadarlat-Toms \cite[Theorem~1.1]{DadTom10Ranks}.
\end{enumerate}

In general, \autoref{qst:range} is still open.
The main result of this paper provides an affirmative answer under the assumption that $A$ has stable rank one:

\begin{thm*}[{\ref{prp:realizeRankCa}}]
Let $A$ be a separable, unital, simple, non-elementary \ca{} with stable rank one. 
Then for every $f\in\LAff(\QT_1(A))_{++}$ there exists $x\in (A\otimes\KK)_+$ with $\rk(x)=f$.
\end{thm*}

An important regularity property in the theory of simple \ca{s} is tensorial absorption of the Jiang-Su algebra $\mathcal{Z}$:
A \ca{} $A$ is said to be \emph{$\mathcal{Z}$-stable} if $A\cong A\otimes\mathcal{Z}$.
This is the \ca{ic} analogue of being a McDuff factor.
The Jiang-Su algebra is a separable, unital, simple, non-elementary \ca{} that is $KK$-equivalent to $\CC$.

By \cite[Theorem~6.7]{Ror04StableRealRankZ}, every unital, simple, stably finite, $\mathcal{Z}$-stable \ca{} has stable rank one.
Thus, \autoref{prp:realizeRankCa} generalizes the results in \cite{EllRobSan11Cone} and \cite{BroPerTom08CuElliottConj} mentioned above.
However, the assumption of stable rank one is much less restrictive than that of $\mathcal{Z}$-stability, and it can in fact be used to \emph{prove} $\mathcal{Z}$-stability;
see for example \autoref{prp:TWlocFinNuclDim}.

Moreover, vast classes of naturally occurring simple, non-$\mathcal{Z}$-stable \ca{s} have stable rank one, including a wealth of non-classifiable, nuclear \ca{s}, \cite{EllHoTom09ClassSimpleSR1, Vil98SimpleCaPerforation, Tom08ClassificationNuclear}, reduced group \ca{s} of free products, \cite{DykHaaRor97SRFreeProd}, and crossed products of dynamical systems with a Cantor factor, \cite{ArcPhi15arX:SRCentrLargeSub, GioKer10SubshiftsPerf}.

\autoref{prp:realizeRankCa} has important consequences for the structure of simple, nuclear \ca{s} with stable rank one.
In particular, it provides the first verification of the Toms-Winter conjecture for a large class of \ca{s} without restriction on the geometry of the simplex of traces.
The Toms-Winter conjecture predicts that three regularity properties of very different natures are equivalent for a separable, unital, simple, non-elementary, nuclear \ca{} $A$:
\begin{enumerate}
\item
$A$ has finite nuclear dimension.
\item
$A$ is $\mathcal{Z}$-stable.
\item
$A$ has strict comparison of positive elements.
\end{enumerate}
We discuss these conditions and the previously known implications in \autoref{sec:TW}.
The condition of finite nuclear dimension is of great importance since it leads to classification by $K$-theoretic and tracial data.
The completion of the classification program is one of the great achievements in operator algebras, obtained in a series of remarkable breakthroughs in the last three years, building on an extensive body of work over decades by numerous people;
see \cite{EllGonLinNiu15arX:classFinDR2} and \cite[Corollary~D]{TikWhiWin17QDNuclear}.

However, in applications it is sometimes difficult to verify finite nuclear dimension directly.
A positive solution of the Toms-Winter conjecture allows finite nuclear dimension to be deduced from relatively simple, verifiable conditions such as strict comparison of positive elements.
We obtain the following partial verifications of the Toms-Winter conjecture:

\begin{thm*}[{\ref{prp:TWlocFinNuclDim}}]
Let $A$ be a separable, unital, simple, non-elementary \ca{} with stable rank one and locally finite nuclear dimension.
Then $A$ is $\mathcal{Z}$-stable if and only if $A$ has strict comparison of positive elements.
\end{thm*}

\begin{thm*}[{\ref{prp:TW-ASHsr1}}]
The Toms-Winter conjecture holds for approximately subhomogeneous \ca{s} with stable rank one.
\end{thm*}

In particular, we obtain that separable, unital, simple, approximately subhomogeneous \ca{s} with stable rank one and strict comparison of positive elements are classified by $K$-theoretic and tracial data.

Theorems~\ref{prp:TWlocFinNuclDim} and~\ref{prp:TW-ASHsr1} apply to large and natural classes of \ca{s}.
For instance, no nuclear \ca{} is known which does not have locally finite nuclear dimension.
Further, all previous partial verifications of the Toms-Winter conjecture required restrictions on the geometry of the simplex of traces.
The most general results so far assumed that the traces form a Bauer simplex with finite-dimensional extreme boundary.
The novelty of our result is that they have no restrictions on the geometry of the trace simplex.

\subsection*{Methods}
The main tool to obtain the results of this paper is the Cuntz semigroup.
It provides a convenient way to organize the comparison theory of positive elements in a \ca{} into a \pom.
Coward-Elliott-Ivanescu, \cite{CowEllIva08CuInv}, initiated a systematic study of Cuntz semigroups by introducing the category $\CatCu$ of abstract Cuntz semigroups, also called \CuSgp{s}.
We refer to Subsections~\ref{sec:Cu} and~\ref{sec:CuA} for details.

Let us explain how to translate \autoref{qst:range} into the setting of \CuSgp{s}.
Given a separable, unital, simple, non-elementary, stably finite \ca{} $A$, its Cuntz semigroup $S:=\Cu(A)$ is a countably based, simple, non-elementary, stably finite \CuSgp{} satisfying certain axioms \axiomO{5} and \axiomO{6}, and the class $u:=[1]$ is a compact, full element in $S$.
The simplex $K$ of normalized $2$-quasitraces on $A$ can be identified with the space of functionals $\lambda\colon S\to[0,\infty]$ that satisfy $\lambda(u)=1$;
see \autoref{pgr:qt}.
Given $a\in S$, we obtain $\rk(a)\in\LAff(K)$ defined by $\rk(a)(\lambda):=\lambda(a)$.
For $x\in(A\otimes\KK)_+$, we have $\rk(x)=\rk([x])$.
Thus, an affirmative answer to \autoref{qst:range} follows directly from one to the following:

\begin{qst}
\label{qst:CuRange}
Let $S$ be a countably based, simple, non-elementary, stably finite \CuSgp{} satisfying \axiomO{5} and \axiomO{6}, let $u\in S$ be a compact, full element, let $K$ denote the simplex of functionals on $S$ that are normalized at $u$, and let $f\in\LAff(K)_{++}$.
Does there exist $a\in S$ with $\rk(a)=f$?
\end{qst}

Let us outline our approach to answer \autoref{qst:CuRange}.
Consider the set $R:=\{\rk(a):a\in S, a\neq 0\}$ of ranks that are realized by nonzero elements in $S$.
We want to verify $R=\LAff(K)_{++}$.
For this, we use the basic order theoretic properties of $\LAff(K)$, which we study in \autoref{sec:LAff}.
We proceed in four steps:

First, we show that $S$ `realizes chisels':
Given $\lambda$ in the extreme boundary $\partial_e K$, we let $\sigma_\lambda\in\LAff(K)$ be the function that takes value $0$ at $\lambda$ and value $\infty$ elsewhere.
Given $t>0$, we call $t+\sigma_\lambda$ the \emph{chisel} at $\lambda$ with value $t$;
see \autoref{dfn:chisel}.
Under certain conditions on $S$, we have $t+\sigma_\lambda\in R$, for every $\lambda\in\partial_e K$ and $t>0$;
see \autoref{prp:realizeChisel}.
This step does not require stable rank one;
see \autoref{prp:realizeChiselCa}.

Second, we show that $S$ `realizes functional infima':
If $f,g\in R$, then $f\wedge g\in R$;
see \autoref{prp:fctlInf}.
It is at this step that the assumption of stable rank one is needed.
The Cuntz semigroups of \ca{s} with stable rank one satisfy a certain axiom \axiomO{6+} (explained below).
We use \axiomO{6+} to realize functional infima.
(In \cite{AntPerRobThi18arX:CuntzSR1} we improve this result by showing that Cuntz semigroups of separable \ca{s} with stable rank one admit general infima, that is, they are inf-semilattices.)

Third, we show that $S$ `realizes ranks approximately':
Under certain conditions on $S$, for all $g\in\Aff(K)_{++}$ and $\varepsilon>0$ there exists $f\in R$ with $g\leq f\leq g+\varepsilon$;
see \autoref{prp:realizeRankApprox}.
The basic idea to obtain this, is to approximate $g$ by infima of chisels.

Fourth, we show that $R$ is closed under passing to suprema of increasing sequences.
Combined with the approximate realization of ranks from step three, we obtain $R=\LAff(K)_{++}$;
see \autoref{prp:realizeRank}.

To summarize, our approach to realizing a given rank is to approximate it by infima of chisels.
This should be contrasted with the approach in \cite{DadTom10Ranks}, where the basic idea is to approximate a given rank by step functions, that is, by sums of scalar multiples of characteristic functions.
\\

The main novel techniques of this paper are two new properties for \CuSgp{s}.
In \autoref{dfn:Edwards}, and inspired by \cite{Edw69UniformApproxAff}, we introduce Edwards' condition, which roughly says that the functional infimum of two element $a$ and $b$ can be pointwise approximated by elements dominated by $a$ and $b$.
Using \AW-completions, we verify Edwards' condition for Cuntz semigroups of all unital \ca{s};
see \autoref{prp:EdwardsCa}.

In \autoref{dfn:O6+}, we introduce axiom \axiomO{6+}, a strengthened version of axiom \axiomO{6} of almost Riesz decomposition.
We show that axiom \axiomO{6+} is satisfied by Cuntz semigroups of \ca{s} with stable rank one (\autoref{prp:RieszSR1}), but not for all \ca{s} (\autoref{exa:notO6}).

In \autoref{sec:chisel}, we show that $S$ realizes chisels if it satisfies Edwards' condition.
This does not require the assumption of stable rank one, and in fact chisels are realized as the ranks of operators in all simple \ca{s};
see \autoref{prp:realizeChiselCa}.
In \autoref{sec:fctlInf}, we show that $S$ realizes functional infima if it satisfies \axiomO{6+}.

\section*{Acknowledgements}

I thank Nate Brown, Tim de Laat, Leonel Robert, Stuart White and Wilhelm Winter for valuable comments.
I am grateful to Eusebio Gardella for his feedback on the first draft of this paper.
Further, I want to thank the anonymous referee for many helpful suggestions.

\section{Preliminaries}
\label{sec:prelim}

\subsection{The category \texorpdfstring{$\CatCu$}{Cu} of abstract Cuntz semigroups}
\label{sec:Cu}

\begin{pgr}
\label{pgr:interpolation}
Let $S$ be a partially ordered set.
Recall that $S$ is said to satisfy \emph{Riesz interpolation} if for all $x,y,a,b\in S$ with $x,y\leq a,b$, there exists $z\in S$ such that $x,y\leq z\leq a,b$.
If for all $a,b\in S$ the infimum $a\wedge b=\sup\{x:x\leq a,b\}$ exists, then $S$ is called an \emph{inf-semilattice}.
A subset $D\subseteq S$ is said to be \emph{upward directed} (\emph{downward directed}) if for every $a,b\in D$ there exists $c\in D$ with $a,b\leq c$ ($c\leq a,b$).
Recall that $S$ is said to be \emph{directed complete}, or a \emph{dcpo} for short, if every directed subset $D$ of $S$ has a supremum $\sup D$;
see \cite[Definition~O-2.1, p.9]{GieHof+03Domains}.

If $S$ is an inf-semilattice then $S$ satisfies Riesz interpolation.
Indeed, given $x,y\leq a,b$, we have $x,y\leq a\wedge b\leq a,b$.
The converse holds if $S$ is directed complete:
Given $a,b\in S$, it follows from Riesz interpolation that the set $\{x:x\leq a,b\}$ is directed, whence directed completeness implies that $a\wedge b$ exists.

Given a dcpo $S$ and $a,b\in S$, recall that $a$ is said to be \emph{way-below} $b$, or that $a$ is \emph{compactly contained} in $b$, denoted $a\ll b$, if for every increasing net $(b_j)_j$ in $S$ with $b\leq\sup_j b_j$ there exists $j$ such that $a\leq b_j$;
see \cite[Definition~I-1.1, p.49]{GieHof+03Domains}.
A \emph{domain} is a dcpo $S$ such that for every $a\in S$ the set $\{x\in S :x\ll a\}$ is upward directed and has supremum $a$;
see \cite[Definition~I.1.6, p.54]{GieHof+03Domains}.

In the context of a \CuSgp{s}, the symbol `$\ll$' is used to denote the sequential way-below relation;
see \autoref{pgr:Cu}.
In general, the way-below relation in a dcpo is stronger than its sequential version, and it may be strictly stronger.
Nevertheless, it will always be clear from context to which relation the symbol `$\ll$' refers.
Moreover, we will now see that under suitable `separability' assumptions, the two notions agree.

Let $S$ be a domain.
A subset $B\subseteq S$ is called a \emph{basis} if for every $a',a\in S$ with $a'\ll a$ there exists $b\in B$ with $a'\leq b\leq a$.
Equivalently, every element in $S$ is the supremum of a directed net in $B$;
see \cite[Definition~III-4.1, Proposition~III-4.2, p.240f]{GieHof+03Domains}.
In particular, $S$ is said to be \emph{countably based} if it contains a countable basis.
In this case, the way-below relation agrees with its sequential version and every element in $S$ is the supremum of a $\ll$-increasing sequence.
\end{pgr}

\begin{pgr}
\label{pgr:pom}
A \emph{partially ordered monoid} is a commutative monoid $M$ together with a partial order $\leq$ such that $a\leq b$ implies $a+c\leq b+c$, for all $a,b,c\in M$.
The partial order on $M$ is said to be \emph{positive} if we have $0\leq a$ for every $a\in M$.
The partial order is said to be \emph{algebraic} if for all $a,b\in M$ we have $a\leq b$ if and only if there exists $c\in M$ with $a+c=b$.
Accordingly, we speak of \emph{positively ordered monoids} and \emph{algebraically ordered monoids}.

A \emph{cone} is a commutative monoid $C$ together with a scalar multiplication by $(0,\infty)$, that is, with a map $(0,\infty)\times C\to C$, denoted $(t,\lambda)\mapsto t\lambda$, that is additive in each variable, and such that $(st)\lambda=s(t\lambda)$ and $1\lambda=\lambda$, for all $s,t\in(0,\infty)$ and $\lambda\in C$.
We do not define scalar multiplication with $0$ in cones.
An \emph{ordered cone} is a cone together with a partial order $\leq$ such that $\lambda_1\leq\lambda_2$ implies $\lambda_1+\mu\leq\lambda_2+\mu$ and $t\lambda_1\leq t\lambda_2$, for all $\lambda_1,\lambda_2,\mu\in C$ and $t\in(0,\infty)$.
If the partial order is positive (algebraic), then we speak of a \emph{positively (algebraically) ordered cone}.

Let $M$ be a partially ordered monoid.
Then $M$ is said to satisfy \emph{Riesz decomposition} if for all $a,b,c\in M$ with $a\leq b+c$ there exist $a_1,a_2\in M$ such that $a=a_1+a_2$, $a_1\leq b$ and $a_2\leq c$.
Further, $M$ is said to satisfy \emph{Riesz refinement} if for all $a_1,a_2,b_1,b_2\in M$ with $a_1+a_2=b_1+b_2$ there exist $x_{i,j}\in M$, for $i,j=1,2$, such that $a_i=x_{i,1}+x_{i,2}$ for $i=1,2$, and $b_j=x_{1,j}+x_{2,j}$ for $j=1,2$.
Among cancellative and algebraically ordered monoids, the properties Riesz decomposition, Riesz refinement, and Riesz interpolation are equivalent.

Further, $M$ is said to be \emph{inf-semilattice-ordered} if $M$ is an inf-semilattice where addition is distributive over $\wedge$, that is,
\begin{align*}
a+(b\wedge c)=(a+b)\wedge(a+c),
\end{align*}
for all $a,b,c\in M$.
\end{pgr}

\begin{pgr}
\label{pgr:Cu}
In \cite{CowEllIva08CuInv}, Coward, Elliott and Ivanescu introduced the category $\CatCu$ of abstract Cuntz semigroups.
Recall that for two elements $a$ and $b$ in a partially ordered set $S$, one says that $a$ is \emph{way-below} $b$, or that $a$ is \emph{compactly contained in} $b$, denoted $a\ll b$, if for every increasing sequence $(b_n)_n$ in $S$ for which $\sup_n b_n$ exists and satisfies $b\leq\sup_n b_n$, there exists $N$ such that $a\leq b_N$.
This is a sequential version of the usual way-below relation used in lattice theory;
see \autoref{pgr:interpolation}.

An \emph{abstract Cuntz semigroup}, also called a \emph{\CuSgp}, is a \pom{} $S$ (see \autoref{pgr:pom}) satisfying the following axioms:
\begin{enumerate}
\item[\axiomO{1}]
Every increasing sequence in $S$ has a supremum.
\item[\axiomO{2}]
Every element in $S$ is the supremum of a $\ll$-increasing sequence.
\item[\axiomO{3}]
Given $a',a,b',b\in S$ with $a'\ll a$ and $b'\ll b$, we have $a'+b'\ll a+b$.
\item[\axiomO{4}]
Given increasing sequences $(a_n)_n$ and $(b_n)_n$ in $S$, we have $\sup_n(a_n+b_n)=\sup_n a_n + \sup_nb_n$.
\end{enumerate}
Given \CuSgp{s} $S$ and $T$, a map $\varphi\colon S\to T$ is a \emph{\CuMor} if it preserves addition, order, the zero element, the way below relation, and suprema of increasing sequences.
If $\varphi$ is not required to preserve the way-below relation then it is called a \emph{generalized \CuMor}.

We often use the following additional axioms:
\begin{enumerate}
\item[\axiomO{5}]
Given $a',a,b',b,c\in S$ satisfying $a+b\leq c$, $a'\ll a$ and $b'\ll b$, there exists $x\in S$ (the `almost complement') such that $a'+x\leq c\leq a+x$ and $b'\ll x$.
\item[\axiomO{6}]
Given $a',a,b,c\in S$ satisfying $a'\ll a\leq b+c$, there exist $e,f\in S$ such that $a'\leq e+f$, $e\leq a,b$ and $f\leq a,c$.
\end{enumerate}
Axiom \axiomO{5} means that $S$ has `almost algebraic order', and axiom \axiomO{6} means that $S$ has `almost Riesz decomposition'.

Recall that a \CuSgp{} $S$ is said to have \emph{weak cancellation}, or to be \emph{weakly cancellative}, if for all $a,b,x\in S$, if $a+x\ll b+x$ then $a\ll b$.
This is equivalent to requiring that for all $a,b,x\in S$, if $a+x\ll b+x$ then $a\leq b$.
It is also equivalent to requiring that for all $a,b,x',x\in S$ satisfying $a+x\leq b+x'$ and $x'\ll x$ we have $a\leq b$.
We refer to \cite[Section~4]{AntPerThi18:TensorProdCu} for details.

A \CuSgp{} $S$ is said to be \emph{countably based} if there exists a countable subset $B\subseteq S$ such that every element in $S$ is the supremum of an increasing sequence with elements in $B$.
For a discussion and proof of the following basic result we refer to \cite[Remarks~3.1.3, p.21f]{AntPerThi18:TensorProdCu}.
\end{pgr}

\begin{prp}
\label{prp:ctblBasedDCPO}
Let $S$ be a countably based \CuSgp.
Then $S$ a domain (see \autoref{pgr:interpolation}) where the way-below relation agrees with its sequential version.
In particular, every upward directed subset $D\subseteq S$ has a supremum.
\end{prp}

\begin{pgr}
\label{pgr:CuFurther}
Let $S$ be a \CuSgp.
A \emph{sub-\CuSgp} of $S$ is a submonoid $T\subseteq S$ such that:
$T$ is closed under passing to suprema of increasing sequences;
$T$ is a \CuSgp{} for the inherited partial order;
and the inclusion map $T\to S$ is a \CuMor.
Given a sub-\CuSgp{} $T\subseteq S$, the way-below relation in $T$ agrees with the restriction of the way-below relation in $S$ to $T$.

An \emph{ideal} in $S$ is a submonoid $J\subseteq S$ that is closed under passing to suprema of increasing sequences and that is downward hereditary, that is, if $a\in S$ and $b\in J$ satisfy $a\leq b$, then $a\in J$.
Every ideal is a sub-\CuSgp, and a sub-\CuSgp{} is an ideal if and only if it is downward hereditary.
We call $S$ \emph{simple} if it only contains the ideals $\{0\}$ and $S$.
Note that $S$ is simple if and only if for every nonzero $a\in S$ the element $\infty a :=\sup_n na$ is the largest element of $S$.

An element $a\in S$ is said to be \emph{finite} if $a\neq a+b$ for every $b\neq 0$.
Further, $S$ is \emph{stably finite} if $a\in S$ is finite whenever there exists $\tilde{a}\in S$ with $a\ll\tilde{a}$;
see \cite[Paragraph~5.2.2]{AntPerThi18:TensorProdCu}.
Every weakly cancellative, simple \CuSgp{} is stably finite.

A \emph{functional} on $S$ is a generalized \CuMor{} $S\to[0,\infty]$.
We use $F(S)$ to denote the set of functionals on $S$.
Equipped with pointwise addition and order, $F(S)$ has the structure of a \pom.
A scalar multiple of a functional is again a functional, which gives $F(S)$ the structure of a positively ordered cone.
If $S$ satisfies \axiomO{5} and \axiomO{6}, then $F(S)$ is an algebraically inf-semilattice-ordered cone;
see Proposition~2.2.3 and Theorem~4.1.2 in \cite{Rob13Cone}.

It was shown in \cite[Theorem~4.8]{EllRobSan11Cone} that $F(S)$ has a natural compact, Hausdorff topology such that a net $(\lambda_j)_j$ in $F(S)$ converges to $\lambda\in F(S)$ if and only if for all $a',a\in S$ with $a'\ll a$ we have
\[
\limsup_j\lambda_j(a')
\leq \lambda(a)
\leq \liminf_j\lambda_j(a).
\]

Given a \pom{} $M$, we let $M^*$ denote the order-preserving monoid morphisms $M\to[0,\infty]$.
Equipped with pointwise order and addition and the obvious scalar multiplication, $M^*$ has the structure of a positively ordered cone.
If $C$ is a positively ordered cone, then every $f\in C^*$ is automatically \emph{homogeneous} (that is $f(t\lambda)=tf(\lambda)$ for $t\in(0,\infty)$ and $\lambda\in C$).
We also say that $f$ is a \emph{linear functional} on $C$.
We set
\[
F(S)^*_\txtLsc := \big\{ f\in F(S)^* : f \text{ lower semicontinuous} \big\}.
\]
Given $a\in S$, we let $\widehat{a}\colon F(S)\to[0,\infty]$ be given by $\widehat{a}(\lambda):=\lambda(a)$, for $\lambda\in F(S)$.
Then $\widehat{a}$ belongs to $F(S)^*_\txtLsc$.
We call $\widehat{a}$ the \emph{rank} of $a$.
If $a\ll a$, then $\widehat{a}$ is continuous.
Given $u\in S$, set
\[
F_{u\mapsto 1}(S) := \big\{ \lambda\in F(S) : \lambda(u)=1 \big\}.
\]
Then $F_{u\mapsto 1}(S)$ is a convex subset of $F(S)$.
Moreover, $F_{u\mapsto 1}(S)$ is closed if and only if $\widehat{u}$ is continuous (for example, if $u$ is compact, that is, $u\ll u$).
We refer to \cite[Theorem~5.2.6, p.50]{AntPerThi18:TensorProdCu} for a characterization of when $F_{u\mapsto 1}(S)$ is nonempty.
If $S$ is simple, stably finite, satisfies \axiomO{5} and \axiomO{6} and $u$ is nonzero and compact, then $F_{u\mapsto 1}(S)$ is a Choquet simplex;
see \autoref{prp:FSChoquet}.
\end{pgr}

\begin{pgr}
\label{pgr:soft}
Let $S$ be a \CuSgp, and let $a,b\in S$.
The element $a$ is said to be \emph{compact} if $a\ll a$.
We use $S_c$ to denote the set of compact elements in $S$.
The element $a$ is said to be \emph{stably below} $b$, denoted $a<_sb$, if there exists $n\in\NN$ such that $(n+1)a\leq nb$.
Further, $a$ is said to be \emph{soft} if $a'<_sa$ for every $a'\in S$ satisfying $a'\ll a$;
see \cite[Definition~5.3.1]{AntPerThi18:TensorProdCu}.
We use $S_\txtSoft$ to denote the set of soft elements in $S$.
Set $S_\txtSoft^\times:=S_\txtSoft\setminus\{0\}$.

By \cite[Theorem~5.3.11]{AntPerThi18:TensorProdCu}, $S_\txtSoft$ is a submonoid of $S$ that is closed under passing to suprema of increasing sequences.
Moreover, $S_\txtSoft$ is absorbing in the sense that $a+b$ is soft as soon as $a$ is soft and $b\leq\infty a$.
In particular, if $S$ is simple, then $a+b$ is soft whenever $a$ is soft and nonzero.

A simple \CuSgp{} is said to be \emph{elementary} if it contains a minimal nonzero element.
The typical elementary \CuSgp{} is $\NNbar:=\{0,1,2,\ldots,\infty\}$.
\end{pgr}

Next, we summarize basic results about the structure of simple \CuSgp{s}.
Part~(1) is \cite[Proposition~5.3.16]{AntPerThi18:TensorProdCu}, parts~(2) and~(3) follow from (the proof of) \cite[Proposition~5.3.18]{AntPerThi18:TensorProdCu}, and part~(4) is \cite[Proposition~5.2.1]{Rob13Cone}.

\begin{prp}
\label{prp:structureSimpleCu}
Let $S$ be a simple, non-elementary, stably finite \CuSgp{} satisfying \axiomO{5} and \axiomO{6}.
Then:
\begin{enumerate}
\item
Every nonzero element in $S$ is either soft or compact.
Thus, $S = S_c \sqcup S_\txtSoft^\times$.
\item
$S_\txtSoft$ is a sub-\CuSgp{} of $S$ that is itself a simple, stably finite \CuSgp{} satisfying \axiomO{5} and \axiomO{6}.
\item
For every nonzero $a\in S$ there exists $x\in S_\txtSoft$ with $0\neq x\leq a$.
\item
For every nonzero $a\in S$ and $n\in\NN$ there exists $x\in S$ with $0\neq nx \leq a$.
\end{enumerate}
\end{prp}

Let $S$ be a partially ordered semigroup.
Recall that $S$ is said to be \emph{unperforated} if for all $a,b\in S$ we have $a\leq b$ whenever $na\leq nb$ for some $n\geq 1$.
Further, $S$ is said to be \emph{almost unperforated} if for all $a,b\in S$ we have $a\leq b$ whenever $a<_s b$ (that is, whenever $(n+1)a\leq nb$ for some $n\geq 1$).

\begin{prp}
\label{prp:charAlmUnp}
Let $S$ be a simple, stably finite, non-elementary \CuSgp{} satisfying \axiomO{5} and \axiomO{6}.
Then the following are equivalent:
\begin{enumerate}
\item
$S$ is almost unperforated.
\item
$S_\txtSoft$ is almost unperforated.
\item
The map $\kappa\colon S_\txtSoft\to F(S)^*_\txtLsc$ is an order-embedding, that is, for all $a,b\in S_\txtSoft$ we have $a\leq b$ whenever $\widehat{a}\leq\widehat{b}$.
\end{enumerate}
\end{prp}
\begin{proof}
By \autoref{prp:structureSimpleCu}, $S_\txtSoft$ is a \CuSgp.
It is clear that~(1) implies that $S_\txtSoft$ is almost unperforated, and that~(3) implies that $S_\txtSoft$ is unperforated.
If $S_\txtSoft$ is almost unperforated, then~(3) follows from \cite[Theorem~5.3.12]{AntPerThi18:TensorProdCu}.
To show that~(3) implies~(1), assume that the order on $S_\txtSoft$ is determined by the functionals.
To show that $S$ is almost unperforated, let $a,d\in S$ and $n\in\NN$ satisfy $(n+1)a\leq nd$.
We need to verify $a\leq d$.
We may assume that there exists $\tilde{d}$ with $d\ll\tilde{d}$.
(For every $a'\ll a$ there exists $d'\ll d$ with $(n+1)a'\leq nd'$.
If for all such $a',d'$ we can deduce $a'\leq d'$, then we obtain $a\leq d$.)
We will repeatedly use the results from \autoref{prp:structureSimpleCu}.
We approximate $a$ and $d$ by suitable soft elements.

We have $(2n+2)a\leq (2n)d$.
Choose $x\in S_\txtSoft^\times$ with $(2n+1)x\leq a$.
Set $b:=a+x$.
Then $a\leq b$, and $b$ is soft since $S_\txtSoft^\times$ is absorbing;
see \autoref{pgr:soft}.
Then
\[
(2n+1)b = (2n+1)a + (2n+1)x \leq (2n+2)a.
\]
Choose $y',y\in S_\txtSoft^\times$ with $y'\ll y$ and $(2n+1)y\leq d$.
Apply \axiomO{5} for $y'\ll y \leq d$ to obtain $z\in S$ with $y'+z\leq d\leq y+z$.
Set $c:=y'+z$.
Then $c\leq d$ and $c$ is soft.
We have
\[
(2n+1)d\leq (2n+1)y+(2n+1)z\leq d+(2n+1)z \leq d+(2n+1)c.
\]
Let $\lambda_\infty\in F(S)$ be the largest functional, which satisfies $\lambda_\infty(s)=\infty$ for every nonzero $s\in S$.
Since there exists $\tilde{d}$ with $d\ll\tilde{d}$ and since $S$ is simple, we have $\lambda(d)<\infty$ for every $\lambda\in F(S)$ with $\lambda\neq\lambda_\infty$.
We deduce $(2n)\widehat{d}\leq(2n+1)\widehat{c}$.
Thus,
\[
(2n+1)\widehat{b}
\leq (2n+2)\widehat{a}
\leq 2n\widehat{d}
\leq(2n+1)\widehat{c}.
\]
By assumption, we obtain $b\leq c$, and hence $a\leq b\leq c\leq d$.
\end{proof}

\begin{prp}[{\cite[Lemma~5.4.3, Proposition~5.4.4]{AntPerThi18:TensorProdCu}}]
\label{prp:realizeRankBySoft}
Let $S$ be a countably based, simple, stably finite, non-elementary \CuSgp{} satisfying \axiomO{5} and \axiomO{6}.
Then for every $a\in S_c$ there exists $x\in S_\txtSoft$ with $x<a$ and $\widehat{x}=\widehat{a}$.

If $S$ is weakly cancellative or almost unperforated, then for each $a\in S$ the set $\{x\in S_\txtSoft : x\leq a\}$ contains a largest element.
Moreover, the map $\varrho\colon S\to S_\txtSoft$ given by $\varrho(a):=\max\{x\in S_\txtSoft : x\leq a\}$ is a generalized \CuMor.
\end{prp}
\begin{proof}
Given $a\in S$ set $P_a:=\{x\in S_\txtSoft : a\leq x\}$.
By \cite[Lemma~5.4.3]{AntPerThi18:TensorProdCu}, there exists $x\in P_a$ with $\widehat{a}=\widehat{x}$.
If $S$ is almost unperforated, then such $x$ is the largest element in $P_a$, by \autoref{prp:charAlmUnp}.
If $S$ weakly cancellative, then it follows from \cite[Proposition~5.4.4]{AntPerThi18:TensorProdCu} that $P_a$ contains a largest element.
Hence, $\varrho$ is well-defined and order-preserving.
It is straightforward to verify that $\varrho$ preserves suprema of increasing sequences.
(This also follows as in the proof of \autoref{prp:alpha}, using that $\varrho$ is the upper adjoint of a Galois connection;
see \autoref{rmk:predecessor}.)
To show that $\varrho$ is additive, let $a,b\in S$.
By \autoref{prp:structureSimpleCu}, each of $a$ and $b$ is either soft or compact.
If $a$ and $b$ are soft, there is nothing to show.
If $a\in S_c$ and $b\in S_\txtSoft^\times$, then it follows from \cite[Proposition~5.4.4]{AntPerThi18:TensorProdCu} that $a+b=\varrho(a)+b$, and consequently
\[
\varrho(a+b)
= \varrho(\varrho(a)+b)
= \varrho(a)+b
= \varrho(a)+\varrho(b).
\]
Lastly, it is enough to consider the case that $a$ and $b$ are compact.
We have $\varrho(a)+\varrho(b)\leq\varrho(a+b)$.
Further,
\[
a+\varrho(b) \leq \varrho(a+b) \leq a+b.
\]
Applying \axiomO{5} for $a\ll a\leq \varrho(a+b)$, we obtain $x\in S$ such that $a+x=\varrho(a+b)$.
Then $x\in S_\txtSoft$ and $a+x\leq a+b$.
Using that $S$ is weakly cancellative or almost unperforated (and simple), we obtain $x\leq b$.
Then $x\leq\varrho(b)$ and thus
\[
\varrho(a+b)=a+x\leq a+\varrho(b)=\varrho(a)+\varrho(b),
\]
as desired.
\end{proof}

\begin{rmks}
\label{rmk:predecessor}
(1)
By \autoref{prp:realizeRankBySoft}, a function in $F(S)^*_\txtLsc$ is realized as $\widehat{a}$ for some $a\in S$ if and only if it is realized as the rank of some soft element.

(2)
Given $a\in S_c$ nonzero, the element $\varrho(a)$ is a `predecessor' of $a$.
We therefore call the map $\varrho\colon S\to S_\txtSoft$ from \autoref{prp:realizeRankBySoft} the \emph{predecessor map}.
Let $\iota\colon S_\txtSoft\to S$ denote the inclusion map.
Given $a\in S_\txtSoft$ and $b\in S$, we have $a\leq\varrho(b)$ if and only if $\iota(a)\leq b$.
Thus, $\varrho$ and $\iota$ form a (order theoretic) Galois connection between $S$ and $S_\txtSoft$;
see \cite[Definition~O-3.1, p.22]{GieHof+03Domains}.
In \autoref{pgr:GaloisAlpha} we show that there is also a Galois connection between $\LAff(K)_{++}$ and $S_\txtSoft^\times$.
\end{rmks}

Let $S$ be a partially ordered semigroup.
Recall that $S$ is said to be \emph{divisible} if for all $a\in S$ and $k\geq 1$ there exists $x\in S$ such that $kx=a$.
If $S$ is a \CuSgp, then is said to be \emph{almost divisible} if for all $a',a\in S$ with $a'\ll a$ and all $k\in\NN$ there exists $x\in S$ such that $kx\leq a$ and $a'\leq(k+1)x$;
see \cite[Definition~7.3.4]{AntPerThi18:TensorProdCu}.

Given a compact, convex set $K$, we use $\LAff(K)_{++}$ to denote the strictly positive, lower semicontinuous, affine functions $K\to[0,\infty]$;
see \autoref{pgr:LAff}.
In the following result, $K$ is a Choquet simplex by \autoref{prp:FSChoquet}.

\begin{prp}
\label{prp:charAlmDiv}
Let $S$ be a countably based, simple, stably finite, non-elementary \CuSgp{} satisfying \axiomO{5} and \axiomO{6}, let $u\in S$ be a compact, full element, and set $K:=F_{u\mapsto 1}(S)$.
Assume that $S$ is almost unperforated.
Then the following are equivalent:
\begin{enumerate}
\item
$S$ is almost divisible.
\item
$S_\txtSoft$ is (almost) divisible.
\item
The map $\kappa\colon S_\txtSoft\to F(S)^*_\txtLsc$ is an order-isomorphism.
\item
For every $f\in\LAff(K)_{++}$, there exists $a\in S_\txtSoft$ with $\widehat{a}_{|K}=f$.
\item
For every $g\in\Aff(K)_{++}$ and $\varepsilon>0$ there exists $a\in S$ with $g\leq\widehat{a}_{|K}\leq g+\varepsilon$.
\end{enumerate}
\end{prp}
\begin{proof}
By \autoref{prp:structureSimpleCu}, $S_\txtSoft$ is a simple \CuSgp{} satisfying \axiomO{5} and \axiomO{6}.
We will use the predecessor map $\varrho\colon S\to S_\txtSoft$ from \autoref{prp:realizeRankBySoft}.
Since $S$ is almost unperforated, so is $S_\txtSoft$, by \autoref{prp:charAlmUnp}.
We will frequently use that the map $\kappa\colon S_\txtSoft\to F(S)^*_\txtLsc$ is an order-embedding.
Since $S$ is simple, we obtain that $K$ is a compact base for the cone $F(S)\setminus\{\lambda_\infty\}$.
It follows that the restriction map $F(S)^*_\txtLsc\to\LAff(K)$, $f\mapsto f_{|K}$, defines an order-isomorphism $F(S)^*_\txtLsc\cong\LAff(K)_{++}\cup\{0\}$.
This shows that~(3) and~(4) are equivalent.
It is clear that~(4) implies~(5).

To show that~(1) implies~(2), assume that $S$ is almost divisible.
We first show that $S_\txtSoft$ is almost divisible.
Let $a',a\in S_\txtSoft$ satisfy $a'\ll a$ and let $k\in\NN$.
Since $S$ is almost divisible, we obtain $x\in S$ with $kx\leq a$ and $a'\leq(k+1)x$.
Set $y:=\varrho(x)$.
Then $ky\leq a$ and $a'\leq(k+1)y$, which shows that $S_\txtSoft$ is almost divisible.
It follows from \cite[Theorem~7.5.4]{AntPerThi18:TensorProdCu} that $S_\txtSoft$ is divisible.

To show that~(2) implies~(3), assume that $S_\txtSoft$ is divisible.
Applying \cite[Theorem~7.5.4]{AntPerThi18:TensorProdCu}, it follows that $S_\txtSoft$ has `real multiplication' in the sense of Robert, \cite[Definition~3.1.2]{Rob13Cone}.
Let $L(F(S))\subseteq F(S)^*_\txtLsc$ be defined as in \cite[Section~3.2]{Rob13Cone}.
Using that $S$ is simple and hence the cone $F(S)\setminus\{\lambda_\infty\}$ has a compact base (namely $K$), we deduce $L(F(S))= F(S)^*_\txtLsc$.
The inclusion $S_\txtSoft\to S$ induces a natural identification $F(S)\cong F(S_\txtSoft)$.
We obtain a natural order-isomorphism $L(F(S_\txtSoft))\cong F(S)^*_\txtLsc$.
It follows from \cite[Theorem~3.2.1]{Rob13Cone} that $S_\txtSoft\to F(S_\txtSoft)^*_\txtLsc$, $a\mapsto\widehat{a}$, defines an order-isomorphism $S_\txtSoft\to L(F(S_\txtSoft))$, which implies~(3).

To show that~(5) implies~(4), one proceeds as in the proof of \autoref{prp:realizeRank}:
Given $f\in\LAff(K)_{++}$, choose an increasing sequence $(g_n)_n$ in $\Aff(K)_{++}$ and a decreasing sequence $(\varepsilon_n)_n$ of positive numbers such that $g_0-\varepsilon_0\in\Aff(K)_{++}$ and $f=\sup_n(g_n-\varepsilon_n)$.
By assumption, for reach $n$ we obtain $a_n\in S$ with $g_n-\varepsilon_n\leq\widehat{a_n}_{|K}\leq g_n-\varepsilon_{n+1}$.
Set $b_n:=\varrho(a_n)$.
Then
\[
\widehat{b_n}_{|K}
= \widehat{a_n}_{|K}
\leq g_n-\varepsilon_{n+1}
\leq g_{n+1}-\varepsilon_{n+1}
\leq \widehat{a_{n+1}}_{|K}
= \widehat{b_{n+1}}_{|K},
\]
for each $n$.
Since $b_n$ and $b_{n+1}$ are soft, we deduce $b_n\leq b_{n+1}$.
Thus, the sequence $(b_n)_n$ is increasing.
Set $b:=\sup_n b_n$.
Then $\widehat{b}_{|K}=f$.

To show that~(3) implies~(2), let $a\in S_\txtSoft$ and let $k\in\NN$.
We have $\tfrac{1}{k}\widehat{a}\in F(S)^*_\txtLsc$.
By assumption, there exists $x\in S_\txtSoft$ with $\widehat{x}=\tfrac{1}{k}\widehat{a}$.
Then $k\widehat{x}=\widehat{a}$ and thus $kx=a$.

To show that~(2) implies~(1), assume that $S_\txtSoft$ is divisible.
Let $a',a\in S$ satisfy $a'\ll a$ and let $k\in\NN$.
By assumption, we obtain $x\in S_\txtSoft$ with $kx=\varrho(a)$.
Then $kx = \varrho(a) \leq a$.
By \cite[Proposition~5.4.4]{AntPerThi18:TensorProdCu}, we have $a+x = \varrho(a)+x$.
Using this at the third step, we obtain
\[
a'\leq a \leq a+x = \varrho(a)+x =(k+1)x,
\]
as desired.
\end{proof}

\subsection{The Cuntz semigroup of a C*-algebra}
\label{sec:CuA}

\begin{pgr}
\label{pgr:CuA}
Let $A$ be a \ca{}.
We use $A_+$ to denote the positive elements in $A$.
Given $a\in A_+$ and $\varepsilon>0$, we write $a_\varepsilon$ for the `$\varepsilon$-cut-down' $(a-\varepsilon)_+$, obtained by applying continuous functional calculus for the function $f(t)=\max\{0,t-\varepsilon\}$ to $a$.

Let $a,b\in A_+$.
We write $a\precsim b$ if $a$ is \emph{Cuntz subequivalent} to $b$, that is, if there exists a sequence $(r_n)_n$ in $A_+$ such that $\lim_n\|a-r_nbr_n^*\|=0$.
We write $a\sim b$ if $a$ is \emph{Cuntz equivalent} to $b$, that is, if $a\precsim b$ and $b\precsim a$.
We write $a\subseteq b$ if $a$ belongs to $\overline{bAb}$, the hereditary sub-\ca{} of $A$ generated by $b$.

R{\o}rdams's lemma states that $a\precsim b$ if and only if for every $\varepsilon>0$ there exist $\delta>0$ and $x\in A$ such that
\[
a_\varepsilon = xx^*, \andSep x^*x \subseteq b_\delta.
\]
It follows that $a\precsim b$ if and only if $a_\varepsilon\precsim b$ for every $\varepsilon>0$.
If $A$ has stable rank one, then $a\precsim b$ if and only if for every $\varepsilon>0$ there exists a unitary $u$ in $\widetilde{A}$, the minimal unitization of $A$, such that $ua_\varepsilon u^*\subseteq b$;
see \cite[Proposition~2.4]{Ror92StructureUHF2}.

The \emph{Cuntz semigroup} of $A$, denoted $\Cu(A)$, is the set of Cuntz equivalence classes of positive elements in the stabilization $A\otimes\KK$.
The class of $a\in(A\otimes\KK)_+$ in $\Cu(A)$ is denoted by $[a]$.
The relation $\precsim$ induces a partial order on $\Cu(A)$.
Using a suitable isomorphism $\varphi\colon A\otimes\KK\otimes M_2\xrightarrow{\cong} A\otimes\KK$, one introduces an addition on $\Cu(A)$ by setting $[a]+[b]:=\left[\varphi\begin{psmallmatrix}a&0\\0&b\end{psmallmatrix}\right]$.
This gives $\Cu(A)$ the structure of a \pom.
It was shown in \cite{CowEllIva08CuInv} that $\Cu(A)$ is always a \CuSgp{}.
\end{pgr}

The next result summarizes the basic properties of Cuntz semigroups of \ca{s}.
We refer to \cite[Section~4]{AntPerThi18:TensorProdCu} for details.
In \autoref{sec:O6+}, we show that $\Cu(A)$ satisfies a strengthened version of \axiomO{6} if $A$ has stable rank one.

\begin{prp}
\label{prp:propertiesCuA}
Let $A$ be a \ca{}.
Then $\Cu(A)$ is a \CuSgp{} satisfying \axiomO{5} and \axiomO{6}.
If $A$ is separable, then $\Cu(A)$ is countably-based.
If $A$ has stable rank one, then $\Cu(A)$ is weakly cancellative.
If $A$ is simple, then $\Cu(A)$ is simple.
\end{prp}

\begin{pgr}
\label{pgr:qt}
The terminology used in connection with quasitraces on \ca{s} is non-uniform across the literature.
We follow \cite[Subsection~2.9]{BlaKir04PureInf}.

Let $A$ be a \ca{}.
A \emph{quasitrace} on $A$ is a map $\tau\colon A_+\to[0,\infty]$ with $\tau(0)=0$, satisfying $\tau(xx^*)=\tau(x^*x)$ for every $x\in A$, and such that $\tau(a+b)=\tau(a)+\tau(b)$ for all $a,b\in A_+$ that commute.
A quasitrace $\tau$ is \emph{bounded} if $\tau(a)<\infty$ for all $a\in A_+$, in which case $\tau$ extends canonically to a (non-linear) functional $A\to\CC$ satisfying $\tau(a)=\tau(a_+)-\tau(a_-)$ for all self-adjoint $a\in A$, and satisfying $\tau(a+ib)=\tau(a)+i\tau(b)$ for all self-adjoint $a,b\in A$.

A \emph{$2$-quasitrace} on $A$ is a quasitrace that extends to a quasitrace $\tau_2$ on $A\otimes M_2$ with $\tau_2(a\otimes e_{1,1})=\tau(a)$ for every $a\in A_+$, where $e_{1,1}$ is the upper-left rank-one projection.
A quasitrace $\tau$ is said to be \emph{lower semicontinuous} if $\tau(a)=\sup_{\varepsilon>0}\tau(a_\varepsilon)$ for every $a\in A_+$.
A bounded quasitrace is automatically lower semicontinuous.
We let $\QT(A)$ denote the set of lower semicontinuous, $2$-quasitraces on $A$.
Every $\tau\in\QT(A)$ is order-preserving on $A_+$;
see \cite[Remarks~2.27]{BlaKir04PureInf}.
The set $\QT(A)$ has the structure of an algebraically ordered cone and it can be equipped with a natural compact, Hausdorff topology;
see \cite[Section~4.1]{EllRobSan11Cone}
We warn the reader that in \cite{BlaHan82DimFct} quasitraces are assumed to be bounded, and in \cite{Haa14Quasitraces} they are also assumed to be $2$-quasitraces.

Let $\tau\in\QT(A)$.
There is a unique extension of $\tau$ to a lower semicontinuous quasitrace $\tau_\infty\colon (A\otimes\KK)_+\to[0,\infty]$ satisfying $\tau_\infty(a\otimes e)=\tau(a)$ for $a\in A_+$ and for every rank-one projection $e\in\KK$.
Abusing notation, we usually denote $\tau_\infty$ by $\tau$.
The induced \emph{dimension function} $d_\tau\colon (A\otimes\KK)_+\to[0,\infty]$ is given by
\[
d_\tau(a) := \lim_n \tau(a^{1/n}),
\]
for $a\in (A\otimes\KK)_+$.

If $a,b\in(A\otimes\KK)_+$ satisfy $a\precsim b$, then $d_\tau(a)\leq d_\tau(b)$.
Therefore, $d_\tau$ induces a well-defined, order-preserving map $\Cu(A)\to[0,\infty]$, which we also denote by $d_\tau$.
One can show that this map is a functional on $\Cu(A)$.
Moreover, every functional on $\Cu(A)$ arises in this way;
see \cite[Proposition~4.2]{EllRobSan11Cone}.
It follows that the assignment $\tau\mapsto d_\tau$ is an additive bijection between $\QT(A)$ and $F(\Cu(A))$, and it is even an isomorphism of ordered topological cones;
see \cite[Theorem~4.4]{EllRobSan11Cone}.

If $A$ is unital, we let $\QT_{1\mapsto 1}(A)$ denote the set of $2$-quasitraces $\tau$ satisfying $\tau(1)=1$.
If $A$ is stably finite, then $\QT_{1\mapsto 1}(A)$ is a nonempty, compact, convex subset of $\QT(A)$ that has the structure of a Choquet simplex;
see \cite[Theoreom~II.4.4]{BlaHan82DimFct}.
Under the identification $\QT(A)\cong F(\Cu(A))$, the Choquet simplex $\QT_{1\mapsto 1}(A)$ corresponds to $F_{[1]\mapsto 1}(\Cu(A))$.
\end{pgr}

\section{Lower semicontinuous affine functions on Choquet simplices}
\label{sec:LAff}


\begin{pgr}
\label{pgr:LAff}
Let $K$ be a \emph{compact, convex set}, by which we always mean a compact, convex subset of a locally convex, Hausdorff, real topological vector space.
The set of extreme points in $K$ is denoted by $\partial_e K$.
We use $\Aff(K)$ to denote the space of continuous, affine functions $K\to\RR$.
We let $\LAff(K)$ denote the set of lower semicontinuous, affine functions $K\to(-\infty,\infty]$.
We equip $\Aff(K)$ and $\LAff(K)$ with pointwise order and addition.
We let $\Aff(K)_{++}$ and $\LAff(K)_{++}$ denote the subsets of strictly positive elements in $\Aff(K)$ and $\LAff(K)$, respectively.

Recall that a partially ordered set $S$ is called \emph{directed complete}, or a \emph{dcpo} for short, if every (upward) direct subset of $S$ has a supremum; equivalently, every increasing net in $S$ has a supremum;
see \cite[Definition~O-2.1, p.9]{GieHof+03Domains}.
Given an increasing net in $\LAff(K)$, the pointwise supremum is a function in $\LAff(K)$.
This shows that $\LAff(K)$ is a dcpo.
\end{pgr}

\begin{prp}[{\cite[Corollary~I.1.4, p.3]{Alf71CpctCvxSets}}]
\label{prp:LAffApproxByAff}
Let $K$ be a compact, convex set.
Then every element of $\LAff(K)$ is the supremum of an increasing net in $\Aff(K)$.
\end{prp}

\begin{prp}[{Bauer's minimum principle, \cite{Bau58Minimalstellen1}}]
\label{prp:BauerMinPrinciple}
Let $K$ be a compact, convex set, and let $f\colon K\to(-\infty,\infty]$ be a lower semicontinuous, concave function.
Then $\inf f(K) = \min f(\partial_e K)$.
This holds in particular for all $f\in\LAff(K)$.
\end{prp}
\begin{proof}
The result is formulated in German as the Corollary to Satz~2 in \cite{Bau58Minimalstellen1}.
A simplified proof of the dual result (Bauer's maximum pricniple) can be found in \cite[Theorem~25.9, p.102]{Cho69LectAna2}:
Every upper semicontinuous, convex function $K\to\RR$ attains its maximum on $\partial_e K$.
This implies Bauer's minimum principle for lower semicontinuous, concave functions $K\to\RR$.
The general case (allowing the value $\infty$) is clear if $\inf f(X)=\infty$.
Otherwise, the result follows by considering the function given as the pointwise infimum of $f$ and a fixed value $>\inf f(K)$.
\end{proof}

The following result is not as well known.
Since the only reference I could find is in German, I include a short argument for the convenience of the reader.

\begin{prp}[{Bauer, \cite{Bau63SimplexeAbgExtr}}]
\label{prp:LAffOrderExtr}
Let $K$ be a compact, convex set, and let $f,g\in\LAff(K)$.
Then $f\leq g$ if and only if $f_{|\partial_eK}\leq g_{|\partial_eK}$.
\end{prp}
\begin{proof}
Apply \autoref{prp:LAffApproxByAff} to obtain an increasing net $(f_j)_{j\in J}$ in $\Aff(K)$ with supremum $f$.
For each $j\in J$, set $g_j:=g-f_j$.
Since $f_j$ is continuous and affine, $g_j$ belongs to $\LAff(K)$.
Moreover, we have $0\leq g_j(\lambda)$ for every $\lambda\in\partial_e K$.
Applying Bauer's minimum principle (\autoref{prp:BauerMinPrinciple}), we obtain $0\leq g_j$, and thus $f_j\leq g$.
Passing to the supremum over $j$, we deduce $f\leq g$.
\end{proof}

\begin{rmk}
Let $K$ be a metrizable, compact, convex set, and let $f,g\colon K\to\RR$ be affine, lower semicontinuous functions satisfying $f_{|\partial_eK}\leq g_{|\partial_eK}$.
Then the conclusion of \autoref{prp:LAffOrderExtr} may also be obtained as follows:
By Choquet's theorem, \cite[Corollary 1.4.9, p.36]{Alf71CpctCvxSets}, for every $\lambda\in K$ there exists a boundary measure $\mu$ (a Borel probability measure $\mu$ with $\mu(\partial_e K)=1$) with barycenter $\lambda$, that is, such that
\[
h(\lambda) = \int_K h(\lambda)d\mu(\lambda) = \int_{\partial_eK} h(\lambda)d\mu(\lambda),
\]
for all $h\in\Aff(K)$.
The above formula holds also for affine functions $K\to\RR$ of first Baire
class, in particular for $f$ and $g$;
see \cite[p.88]{Phe01LNMChoquet}.
Given $\lambda\in K$, we choose a boundary measure $\mu$ with barycenter $\lambda$, and obtain
\[
f(\lambda)
= \int_{\partial_e K} f(\lambda)d\mu(\lambda)
\leq \int_{\partial_e K} g(\lambda)d\mu(\lambda)
= g(\lambda).
\]
\end{rmk}

Recall the definition of a domain and the way-below relation from \autoref{pgr:interpolation}.

\begin{lma}
\label{prp:llInLAff}
Let $K$ be a compact, convex set.
Then $\LAff(K)$ is a domain.
Given $f,g\in\LAff(K)$, we have $f\ll g$ if and only if there exist $h\in\Aff(K)$ and $\varepsilon>0$ with $f+\varepsilon\leq h\leq g$.
Further, $K$ is metrizable if and only if $\LAff(K)$ is countably based.
\end{lma}
\begin{proof}
It is straightforward to verify that $\LAff(K)$ is a dcpo.
To verify the characterization of the way-below relation, let $f,g\in\LAff(K)$.
We first assume that $f\ll g$.
Apply \autoref{prp:LAffApproxByAff} to obtain an upward directed subset $D\subseteq \Aff(K)$ such that $g=\sup D$.
Set
\[
D' := \big\{ h-\varepsilon : h\in D, \varepsilon>0 \big\}.
\]
It is straightforward to show that $D'$ is upward directed and that $g=\sup D'$.
By definition of the way-below relation, we can choose $h'\in D'$ with $f\leq h'$.
Choose $h\in D$ and $\varepsilon>0$ such that $h'=h-\varepsilon$.
Then $f+\varepsilon\leq h\leq g$, as desired.

Conversely, let $h\in\Aff(K)$ and $\varepsilon>0$ with $f+\varepsilon\leq h\leq g$.
To verify $f\ll g$, let $D$ be an upward directed subset of $\LAff(K)$ with $g\leq\sup D$.
Given $x\in K$, we have
\[
h(x)-\varepsilon < h(x)\leq g(x)=\sup \big\{ g'(x):g'\in D \big\},
\]
which allows us to choose $g_x\in D$ with $h(x)-\varepsilon < g_x(x)$.
Using that $h$ is continuous and that $g_x$ is lower-semicontinuous, we can choose a neighborhood $V_x$ of $x$ such that $h-\varepsilon<g_x$ on $V_x$.
Use that $K$ is compact to choose $x_1,\ldots,x_n$ such that $K=\bigcup_{k=1}^nV_{x_k}$.
Since $D$ is upward directed, we can pick $g'\in D$ with $g_{x_1},\ldots,g_{x_n}\leq g'$.
Then $f\leq g'$, as desired.

Using that every element in $\LAff(K)$ is the supremum of an upward directed subset of $\Aff(K)$, see \autoref{prp:LAffApproxByAff}, and using the above characterization of the way-below relation, it follows that $\LAff(K)$ is a domain.

Using the description of the way-below relation, it follows that $\LAff(K)$ is countably based if and only if $\Aff(K)$ is separable (for the topology induced by the supremum-norm).
Assume that $K$ is metrizable.
Then the Banach algebra $C(K,\RR)$ of continuous functions $K\to\RR$ is separable.
Since $\Aff(K)$ is a closed subspace of $C(K,\RR)$, it is separable as well;
see \cite[Proposition~14.10, p.221]{Goo86GpsInterpolation}.

Conversely, assume that $\Aff(K)$ is separable.
Since the elements of $\Aff(K)$ separate the points of $K$, the family of sets of the form $f^{-1}((0,\infty])$, with $f$ ranging over a countable dense subset of $\Aff(K)$, forms a countable basis for the topology of $K$.
Hence, $K$ is second-countable and therefore metrizable.
\end{proof}


\begin{lma}
\label{prp:LAffRiesz}
Let $K$ be a compact, convex set.
Then the following are equivalent:
\begin{enumerate}
\item
$K$ is a Choquet simplex.
\item
$\LAff(K)$ is an inf-semilattice.
\item
$\LAff(K)$ satisfies Riesz interpolation.
\end{enumerate}
Moreover, if the above hold, then $\LAff(K)$ is inf-semilattice-ordered and we have
\begin{align}
\label{prp:LAffRiesz:eqInfAtExtr}
(f\wedge g)(\lambda)=f(\lambda)\wedge g(\lambda),
\end{align}
for all $f,g\in\LAff(K)$ and $\lambda\in\partial_e K$.
\end{lma}
\begin{proof}
By \cite[Theorem~II.3.8, p.89]{Alf71CpctCvxSets}, $K$ is a Choquet simplex if and only if the upper semicontinuous, affine functions $K\to[-\infty,\infty)$ form a sup-semilattice.
The latter is equivalent to $\LAff(K)$ being an inf-semilattice, which shows the equivalence of~(1) and~(2).
A dcpo is an inf-semilattice if and only if it satisfies Riesz refinement;
see \autoref{pgr:interpolation}.
Using that $\LAff(K)$ is a dcpo, we obtain that~(2) and~(3) are equivalent.

Assume that statements~(1)-(3) are satisfied.
To show \eqref{prp:LAffRiesz:eqInfAtExtr}, let $f,g\in\LAff(K)$.
Consider the function $h\colon K\to(-\infty,\infty]$ given by $h(\lambda):=f(\lambda)\wedge g(\lambda)$, for $\lambda\in K$.
Then $h$ is lower semicontinuous and concave.
Let $\check{h}$ be the lower envelope of $h$;
see \cite[p.4]{Alf71CpctCvxSets}.
It follows from \cite[Theorem~II.3.8]{Alf71CpctCvxSets} that $\check{h}$ is lower semicontinuous and affine.
This shows that $\check{h}$ is the infimum of $f$ and $g$ in $\LAff(K)$.

By \cite[Proposition~I.4.1]{Alf71CpctCvxSets}, we have $\check{k}(\lambda)=k(\lambda)$ for every lower semicontinuous function $k\colon K\to(-\infty,\infty]$ and every $\lambda\in\partial_e K$.
Applied for $h$, we obtain
\[
(f\wedge g)(\lambda)
= \check{h}(\lambda)
= h(\lambda)
= f(\lambda)\wedge g(\lambda),
\]
for all $\lambda\in\partial_e K$. 

To show that $\LAff(K)$ is inf-semilattice-ordered, let $f,g,h\in\LAff(K)$.
Using \eqref{prp:LAffRiesz:eqInfAtExtr} at the first and last step, we obtain
\begin{align*}
(f+(g\wedge h))(\lambda)
&= f(\lambda)+g(\lambda)\wedge h(\lambda) \\
&= (f(\lambda)+g(\lambda))\wedge(f(\lambda)+h(\lambda))
= ((f+g)\wedge(f+h))(\lambda),
\end{align*}
for all $\lambda\in\partial_eK$.
Using \autoref{prp:LAffOrderExtr}, we deduce $f+(g\wedge h)=(f+g)\wedge(f+h)$.
\end{proof}

\begin{lma}
\label{prp:FSChoquet}
Let $S$ be a simple, stably finite \CuSgp{} satisfying \axiomO{5} and \axiomO{6}, and let $u\in S$ be nonzero such that $u\ll a$ for some $a\in S$, and such that $\widehat{u}$ is continuous (for example, $u$ is nonzero and compact).
Then $K:=F_{u\to 1}(S)$ is a nonempty Choquet simplex.
If $S$ is countably based, then $K$ is metrizable.
\end{lma}
\begin{proof}
By Proposition~2.2.3 and Theorem~4.1.2 in \cite{Rob13Cone}, $F(S)$ is an algebraically ordered, complete lattice.
Let $\lambda_\infty\in F(S)$ be the largest functional, which satisfies $\lambda_\infty(a)=\infty$ for all nonzero $a\in S$.
Set $C:=F(S)\setminus\{\lambda_\infty\}$.
Then $C$ is an algebraically ordered, locally compact cone.
Moreover, $C$ is cancellative by the remarks above Proposition~3.2.3 in \cite{Rob13Cone}.
(In the notation used there, we have $C=F_S(S)$.)
It follows that $C$ satisfies Riesz refinement.

Using that $\widehat{u}$ is continuous, we deduce that $K$ is a closed, convex subset of $F(S)$.
Let $\lambda_0$ be the zero functional.
Given $\lambda\in F(S)$ with $\lambda\neq\lambda_0,\lambda_\infty$, we have $\lambda(u)\neq 0$ (using that $u$ is full) and $\lambda(u)\neq\infty$ (using that there exists $a\in S$ with $u\ll a$).
Hence, the functional $\tfrac{1}{\lambda(u)}\lambda$ belongs to $K$ and $\lambda=\lambda(u)\cdot\tfrac{1}{\lambda(u)}\lambda$.
It follows that $K$ is a compact base of $C$.
Since $C$ satisfies Riesz refinement, it follows from \cite[Proposition~II.3.3, p.85]{Alf71CpctCvxSets} that $K$ is a Choquet simplex.

If $S$ is countably based, it follows as in the proof of \autoref{prp:llInLAff} that $K$ is metrizable.
\end{proof}

\begin{prp}
\label{prp:LAffCu}
Let $K$ be a metrizable, compact, convex set.
Consider $S:=\LAff(K)_{++}\cup\{0\}$.
Then $S$ is a weakly cancellative, countably based, simple, stably finite \CuSgp{} satisfying \axiomO{5}. 
Further, the following are equivalent:
\begin{enumerate}
\item
$K$ is a Choquet simplex.
\item
$S$ is semilattice-ordered.
\item
$S$ satisfies \axiomO{6+}. (See \autoref{dfn:O6+}.)
\item
$S$ satisfies \axiomO{6}.
\end{enumerate}
\end{prp}
\begin{proof}
It follows from \autoref{prp:llInLAff} that $S$ is a countably based \CuSgp.
Every nonzero element in $S$ is a strictly positive function on $K$ and therefore full in $S$, which shows that $S$ is simple.
For $a,b\in S$ we have $a\ll b$ in $S$ if and only if $a\ll b$ in $\LAff(K)$.
Using that finite-valued functions in $\LAff(K)$ can be canceled, one shows that $S$ is weakly cancellative and stably finite.

To verify that $S$ satisfies \axiomO{5}, let $a',a,b',b,c\in S$ satisfy
\[
a+b\leq c, \andSep a'\ll a, \andSep b'\ll b.
\]
We need to find $x\in S$ satisfying $a'+x\leq c\leq a+x$ and $b'\ll x$.
If $a'=0$, then $x:=c$ has the desired properties.
If $a'\neq 0$, apply \autoref{prp:llInLAff} to obtain $f\in\Aff(K)_{++}$ and $\varepsilon>0$ such that
\[
a'\leq f, \andSep f+\varepsilon\leq a.
\]
Define $x\colon K\to[0,\infty]$ by $x(\lambda):=c(\lambda)-f(\lambda)$, for $\lambda\in K$.
Since $f\in\Aff(K)$, we have $x\in\LAff(K)$.
It follows that $x$ has the desired properties.

If $K$ is a Choquet simplex, then $\LAff(K)$ is semilattice-ordered by \autoref{prp:llInLAff}, which implies that $S$ is semilattice-ordered as well.
Assume that $S$ is semilattice-ordered.
To verify \axiomO{6+}, let $a,b,c,x',x,y',y\in S$ satisfy
\[
a \leq b+c, \andSep x'\ll x \leq a,b, \andSep y'\ll y \leq a,c.
\]
Set $e:=a\wedge b$ and $f:=a\wedge c$.
We have $x'\ll e\leq a,b$ and $y'\ll f\leq a,c$.
Moreover,
\begin{align*}
e+f
= (a\wedge b)+(a\wedge c)
&= ((a\wedge b)+a)\wedge((a\wedge b)+c) \\
&= (2a)\wedge(a+b)\wedge(a+c)\wedge(b+c)
\geq a,
\end{align*}
which shows that $e$ and $f$ have the desired properties.

It is clear that \axiomO{6+} implies \axiomO{6}.
Lastly, let us assume that $S$ satisfies \axiomO{6}.
Set $u:=1$, the constant function of value $1$ on $K$.
Then $\widehat{u}$ is continuous and $F_{u\mapsto 1}(S)$ is affinely homeomorphic to $K$.
It follows from \autoref{prp:FSChoquet} that $K$ is a Choquet simplex.
\end{proof}

\section{Edwards' condition}
\label{sec:Edwards}

Throughout this section, we let $S$ be a simple, stably finite \CuSgp{} satisfying \axiomO{5} and \axiomO{6}, we let $u\in S$ be a compact, full element, and we set $K:=F_{u\to 1}(S)$, the compact, convex set of normalized functionals.
It follows from \autoref{prp:FSChoquet} that $K$ is a nonempty Choquet simplex.

\begin{dfn}
\label{dfn:Edwards}
Given $\lambda\in\partial_e K$, we say that $S$ satisfies \emph{Edwards' condition for $\lambda$} if
\begin{align}
\label{dfn:Edwards:eq}
\min\{\lambda(a),\lambda(b)\} = \sup \big\{ \lambda(c) : c\leq a,b \big\},
\end{align}
for all $a,b\in S$.
If this holds for all $\lambda\in\partial_e(K)$, then we say that $S$ satisfies \emph{Edward's condition for $\partial_e(K)$}.
\end{dfn}

\begin{rmks}
\label{rmk:Edwards}
(1)
There is natural way to formulate Edwards' condition for arbitrary functionals.
This is especially relevant for studying ranks of operators in non-simple or non-unital \ca{s}.
We will pursue this in forthcoming work.

(2)
Since $K$ is a Choquet simplex, we may apply \eqref{prp:LAffRiesz:eqInfAtExtr} to identify the left hand side in \eqref{dfn:Edwards:eq} with $(\widehat{a}_{|K}\wedge\widehat{b}_{|K})(\lambda)$.

(3)
Let $\lambda\in\partial_e K$.
Then $S$ satisfies Edwards' condition for $\lambda$ if for all $a,b\in S$ and $t\in[0,\infty)$ satisfying $t < \lambda(a),\lambda(b)$, there exists $c\in S$ such that
\[
t<\lambda(c), \andSep c\leq a,b.
\]
Similarly, by \autoref{prp:EdwardsDual}, if $S$ satisfies Edwards' condition for $\lambda$, then for all $a,b\in S$ and $t\in(0,\infty]$ satisfying $\lambda(a),\lambda(b) < t$, there exists $d\in S$ such that
\[
a,b\leq d, \andSep \lambda(d)<t.
\]
These formulations are similar to condition~(2) considered by Edwards in \cite{Edw69UniformApproxAff}, and this is the reason for the terminology in \autoref{dfn:Edwards}.
\end{rmks}

The statement of the following result can be considered as a dual version of Edwards' condition. It will be used in \autoref{prp:realizeChisel}.

\begin{prp}
\label{prp:EdwardsDual}
Let $\lambda\in\partial_e K$ be such that $S$ satisfies Edwards' condition for $\lambda$.
Then
\begin{align*}
\max\{\lambda(a),\lambda(b)\} = \inf \big\{ \lambda(d) : a,b\leq d \big\},
\end{align*}
for all $a,b\in S$.
\end{prp}
\begin{proof}
The inequality `$\leq$' is clear. 
To show the converse, let $a,b\in S$.
Without loss of generality, we may assume that $\lambda(a)\leq\lambda(b)<\infty$.
Let $t\in(0,\infty)$ satisfy $\lambda(b) < t$.
We need to find $d\in S$ such that
\[
a,b\leq d, \andSep \lambda(d)<t.
\]

If $\lambda(a)=0$, then $a=0$ by simplicity, and thus $d:=b$ has the desired properties.
Assume that $0<\lambda(a)$.
Choose $\varepsilon>0$ such that
\[
\lambda(b) < t-2\varepsilon,
\]
and choose $s\in[0,\infty)$ such that
\[
s < \lambda(a) \leq s+\varepsilon.
\]
Applying Edwards' condition, we obtain $c\in S$ such that
\[
s<\lambda(c), \andSep c\leq a,b.
\]
Choose $c'\ll c$ such that
\[
s<\lambda(c'), \andSep \lambda(c)\leq\lambda(c')+\varepsilon.
\]
Applying \axiomO{5} for $c'\ll c\leq a$ and $c'\ll c\leq b$, we obtain $x,y\in S$ such that
\[
c'+x \leq a \leq c+x, \andSep c'+y \leq b \leq c+y.
\]
Set $d:=c+x+y$.
We have
\[
\lambda(c') + \lambda(x) \leq \lambda(a) \leq s + \varepsilon \leq \lambda(c') +\varepsilon.
\]
Since $c'\ll c\leq\infty u$, and $\lambda(u)=1$, we have $\lambda(c')<\infty$, which allows us to cancel it from the above inequality to obtain $\lambda(x)\leq\varepsilon$.
We deduce
\[
\lambda(d)
= \lambda(c)+\lambda(x)+\lambda(y)
\leq \lambda(c')+2\varepsilon+\lambda(y)
\leq \lambda(b)+2\varepsilon
< t.
\]
Further, we have
\[
a \leq c+x \leq c+x+y = d, \andSep
b \leq c+y \leq c+y+x = d,
\]
which shows that $d$ has the desired properties.
\end{proof}

Our next goal is to verify that Cuntz semigroups of \ca{s} satisfy Edwards' condition;
see \autoref{prp:EdwardsCa}.
A main ingredient in the proof are \AW-completions as developed in \cite{Haa14Quasitraces}, which in turn is based on \cite{BlaHan82DimFct}.
We recall some details.

\begin{pgr}
\label{pgr:AWcompl}
Let $A$ be a unital \ca{}.
Recall that $\QT_{1\mapsto 1}(A)$ denotes the normalized $2$-quasitraces;
see \autoref{pgr:qt}.
Let $\tau\in\QT_{1\mapsto 1}(A)$.
Then $\tau$ extends canonically to a (possibly non-linear) functional $A\to\CC$, which we also simply denote by $\tau$.
Given $x\in A$, set
\[
\|x\|_\tau := \tau(x^*x)^{1/2}.
\]
By \cite[Lemma~3.5(2)]{Haa14Quasitraces}, we have
\begin{align}
\label{pgr:qt:eq2}
\|x+y\|^{2/3}_\tau \leq \|x\|^{2/3}_\tau + \|y\|^{2/3}_\tau,
\end{align}
for all $x,y\in A$.

If $x\in A$ is normal, then the restriction of $\tau$ to the sub-\ca{} generated by $x$ is linear, whence we can apply the Cauchy-Schwarz inequality to obtain
\begin{align}
\label{pgr:openProj:eq2CS}
|\tau(x)| \leq \|x\|_\tau.
\end{align}

For general $x$, let $a$ and $b$ be the real and imaginary parts of $x$.
Applying \eqref{pgr:qt:eq2} at the third step, we obtain
\[
\|a\|_\tau^2
= \big\| \frac{1}{2}(x+x^*) \big\|_\tau^2
= \frac{1}{4} \big( \|x+x^*\|_\tau^{2/3} \big)^3
\leq \frac{1}{4} \big( \|x\|_\tau^{2/3} + \|x^*\|_\tau^{2/3} \big)^3
= 2 \|x\|_\tau^2,
\]
and similarly $\|b\|_\tau^2\leq 2\|x\|_\tau^2$.
We deduce that
\begin{align*}
|\tau(x)|
= |\tau(a)+i\tau(b)|
= \big( |\tau(a)|^2 + |\tau(b)|^2 \big)^{\frac{1}{2}}
\leq \big( \|a\|_\tau^2 + \|b\|_\tau^2 \big)^{\frac{1}{2}}
\leq 2\|x\|_\tau.
\end{align*}

Using \eqref{pgr:qt:eq2}, one can show that $K_\tau := \big\{ x: \|x\|_\tau=0 \big\}$ is a closed, two-sided ideal.
If $K_\tau=\{0\}$, then $\tau$ is said to be \emph{faithful}.
If $A$ is simple, then $\tau$ is automatically faithful.
Define $\dist_\tau\colon A\times A\to[0,\infty)$ by setting
\begin{align*}
\dist_\tau(x,y) := \|x-y\|_\tau^{2/3} = \tau\big((x-y)^*(x-y)\big)^{1/3},
\end{align*}
for $x,y\in A$.
Then $\dist_\tau$ is a pseudometric on $A$.
If $\tau$ is faithful, then $\dist_\tau$ is a metric;
see \cite[Lemma~3.5, Definition~3.6]{Haa14Quasitraces}.
(In \cite{Haa14Quasitraces}, $\dist_\tau$ is denoted by $d_\tau$.
We introduce different notation since we reserve $d_\tau$ to denote the dimension function induced by $\tau$.)

Let $\ell^\infty(A)$ denote the \ca{} of bounded sequences in $A$.
Set
\begin{align}
\label{pgr:AWcompl:eq2tildeA}
\tilde{A}
&:= \big\{ (x_n)_n \in \ell^\infty(A) : (x_n)_n \text{ is a $\dist_\tau$-Cauchy sequence} \big\}, \andSep \\
\label{pgr:AWcompl:eq3tildeJ}
\tilde{J}
&:= \big\{ (x_n)_n \in \ell^\infty(A) : x_n\to 0 \text{ in the $\dist_\tau$-pseudometric} \big\},
\end{align}

Set $M:=\tilde{A}/\tilde{J}$, and let $\pi\colon \tilde{A}\to M$ denote the quotient \stHom.
It follows that $\tau$ induces a map $\tilde{A}\to\CC$ by $(x_n)_n\mapsto \lim_{n\to\infty}\tau(x_n)$.
One verifies that this is a normalized $2$-quasitrace on $\tilde{A}$ that vanishes on $\tilde{J}$.
We let $\bar{\tau}\colon M\to\CC$ be the induced map satisfying
\[
\bar{\tau}(\pi(x)) := \lim_{n\to\infty} \tau(x_n),
\]
for $x=(x_n)_n\in\tilde{A}$.
Given $x\in A$, the constant sequence $(x)_n$ belongs to $\tilde{A}$.
We obtain a \stHom{} $\iota\colon A\to M$ by $\iota(x):=\pi((x)_n)$.
We have $\bar{\tau}\circ\iota=\tau$.
If $\tau$ is faithful, then $\iota$ is injective and we can consider $A$ as a sub-\ca{} of $M$.

The algebra $M$ is an \AW-algebra and $\bar{\tau}$ is a faithful, normal, normalized $2$-quasitrace on $M$;
see \cite[Section~4]{Haa14Quasitraces}.
If $\tau$ is an extreme point in $\QT_{1\to 1}(A)$, then $M$ is a finite \AW-factor;
see \cite[Proposition~4.6]{Haa14Quasitraces}.
\end{pgr}

\begin{pgr}
\label{pgr:openProj}
Let $A$ be a unital \ca{}, and let $\tau\in\QT_{1\mapsto 1}(A)$.

Let $e\in A_+$.
Then $e^*e\leq\|e\|e$.
Using at the second step that $\tau$ is order-preserving, we obtain
\begin{align}
\label{pgr:openProj:eq1estNormTau}
\|e\|_\tau = \tau(e^*e)^{1/2} \leq \tau(\|e\|e)^{1/2} = \|e\|^{1/2} \tau(e)^{1/2}.
\end{align}

Let $a\in A_+$ be a contraction.
Then the sequence $\bar{a}:=(a^{1/n})_n$ is bounded.
Let us verify that $\bar{a}$ belongs to $\tilde{A}$, as defined in \eqref{pgr:AWcompl:eq2tildeA}.
Given $m\leq n$, we have $a^{1/m}\leq a^{1/n}$.
Using \eqref{pgr:openProj:eq1estNormTau} and $\|a^{1/n}- a^{1/m}\|\leq 1$ at the second step, and using at the last step that $a^{1/n}$ and $a^{1/m}$ commute, we deduce that
\[
\dist_\tau(a^{1/n},a^{1/m})^{1/3}
= \|a^{1/n}- a^{1/m}\|_\tau^2
\leq \tau(a^{1/n}- a^{1/m})
= \tau(a^{1/n})-\tau(a^{1/m}).
\]
We have $\lim_n\tau(a^{1/n})=d_\tau(a)$, which implies that $(a^{1/n})_n$ is a $\dist_\tau$-Cauchy sequence.
We have verified that $\bar{a}\in\tilde{A}$, which allows us to set
\begin{align*}
p_a:= \pi( \bar{a} ).
\end{align*}
Then
\begin{align}
\label{pgr:openProj:eq4tauPa}
\bar{\tau}(p_a) = \lim_n \tau(a^{1/n}) = d_\tau(a), \andSep
\bar{\tau}(p_a^2) = \lim_n \tau(a^{2/n}) = d_\tau(a).
\end{align}
For each $n$, we have $a^{2/n}\leq a^{1/n}$, and therefore $p_a^2\leq p_a$.
Since $p_a$ and $p_a^2$ commute, and since $\bar{\tau}$ is faithful on $M$, it follows that $p_a^2=p_a$.
Note that $p_a$ is the support projection of $\iota(a)$ in $M$, that is, the smallest projection acting as a unit on $\iota(a)$.
\end{pgr}

\begin{lma}
\label{prp:EdwardsCaPre}
Let $A$ be a unital \ca{}, let $\tau\in\partial_e \QT_{1\mapsto 1}(A)$, and let $a,b\in A_+$.
Then
\begin{align*}
\min\{d_\tau(a),d_\tau(b)\} = \sup \big\{ d_\tau(c) : c\in A_+, c\precsim a,b\big\}.
\end{align*}
\end{lma}
\begin{proof}
The inequality `$\geq$' is clear.
Let us show the converse.
Without loss of generality, we may assume that $a$ and $b$ are contractions, and that $d_\tau(a)\leq d_\tau(b)$.
The statement is clear if $d_\tau(a)=0$.
Thus, we may assume that $0<d_\tau(a)$.
Let $t\in[0,\infty)$ satisfy $t<d_\tau(a)$.
We need to find $c\in A_+$ such that
\[
t<d_\tau(c), \andSep c\precsim a,b.
\]

Let $(M, \bar{\tau})$ be \AW-completion of $(A,\tau)$ as described in \autoref{pgr:AWcompl}.
Since $\tau$ is extremal, $M$ is a factor;
see \cite[Proposition~4.6]{Haa14Quasitraces}.
Consider $\tilde{A}$ as in \eqref{pgr:AWcompl:eq2tildeA}.
As explained in \autoref{pgr:openProj}, we have $\bar{a},\bar{b}\in\tilde{A}$.
Set
\[
p_a:= \pi( \bar{a} ), \andSep
p_b:= \pi( \bar{b} ).
\]
Then $p_a$ and $p_b$ are projections satisfying
\[
\bar{\tau}(p_a)
\stackrel{\eqref{pgr:openProj:eq4tauPa}}{=} d_\tau(a)
\leq d_\tau(b)
\stackrel{\eqref{pgr:openProj:eq4tauPa}}{=} \bar{\tau}(p_b).
\]
The order of projections in an \AW-factor is determined by its normalized quasitrace;
see \cite[Proposition~27.1, p.160]{Ber72BearRgs}.
We may therefore choose $v\in M$ such that $p_a=vv^*$ and $v^*v\leq p_b$.
Lift $v$ to a contractive element $\bar{v}=(v_n)_n$ in $\tilde{A}$.
Set
\[
w_n := a^{1/n}v_nb^{1/n},
\]
for $n\in\NN$, and $\bar{w}:=(w_n)_n$.
Then $\bar{w}=\bar{a}\bar{v}\bar{b}\in\tilde{A}$, and
\[
\pi(\bar{a}\bar{v}\bar{b})=p_avp_b=v.
\]
Thus, $\bar{w}$ is also a contractive lift of $v$.
For each $n$, we have
\[
w_nw_n^* = a^{1/n}v_nb^{2/n}v_n^*a^{1/n}\leq a^{2/n},
\]
and thus $\bar{a}^2-\bar{w}\bar{w}^*\geq 0$.
We deduce that
\[
\tau\big( a^{2/n}-w_nw_n^*\big)
= \big| \tau(a^{2/n}-w_nw_n^*) \big|
\stackrel{\eqref{pgr:openProj:eq2CS}}{\leq}
\big\| a^{2/n} - w_nw_n^* \big\|_\tau,
\]
for each $n$.
We have
\[
\pi(\bar{a}^2-\bar{w}\bar{w}^*) = p_a^2-vv^*=0,
\]
and thus $\bar{a}^2-\bar{w}\bar{w}^*\in\tilde{J}$.
It follows that
\begin{align}
\label{prp:EdwardsCaFactor:eqEstZero}
\lim_n \tau\big( a^{2/n}-w_nw_n^*\big) \leq \lim_n \| a^{2/n} - w_nw_n^* \big\|_\tau = 0.
\end{align}

Using \eqref{pgr:openProj:eq1estNormTau} at the second step, we deduce that
\[
\tau(a^{2/n})^{1/2}
= \tau\big( (a^{2/n}-w_nw_n^*) + w_nw_n^* \big)^{1/2}
\leq \tau\big( a^{2/n}-w_nw_n^* \big)^{1/2} + \tau\big( w_nw_n^* \big)^{1/2}.
\]
Using \eqref{prp:EdwardsCaFactor:eqEstZero} and using that $\lim_n\tau(a^{2/n})=d_\tau(a)>t$, we can choose $n\in\NN$ such that $\tau(w_nw_n^*)> t$.
Since $w_nw_n^*$ is contractive, we have $d_\tau(w_nw_n^*)\geq\tau(w_nw_n^*)>t$.
Moreover,
\[
w_nw_n^* = a^{1/n}v_nb^{2/n}v_n^*a^{1/n} \precsim a,b.
\]
Hence, $c:=w_nw_n^*$ has the desired properties.
\end{proof}

\begin{thm}
\label{prp:EdwardsCa}
Let $A$ be a unital \ca{}. 
Then $\Cu(A)$ satisfies Edwards' condition for $\partial_e F_{[1]\mapsto 1}(\Cu(A))$.
\end{thm}
\begin{proof}
We identify $\QT_{1\to 1}(A)$ with $F_{[1]\to 1}(\Cu(A))$, as in \autoref{pgr:qt}.
Let $\tau\in \partial_e \QT_{1\to 1}(A)$.
We let $\tau_\infty$ denote the (unique) extension of $\tau$ to a lower-semicontinuous $2$-quasitrace on $A\otimes\KK$.
Further, $d_\tau\colon(A\otimes\KK)_+\to[0,\infty]$ denotes the induced dimension function.
Let $a,b\in (A\otimes\KK)_+$, and let $t\in\RR$ satisfy $t<d_\tau(a),d_\tau(b)$.
To verify Edwards' condition, we need to find $c\in(A\otimes\KK)_+$ such that $t<d_\tau(c)$ and $c\precsim a,b$.

We have
\[
d_\tau(a) = \sup_{\varepsilon>0} d_\tau(a_\varepsilon),\andSep
d_\tau(b) = \sup_{\varepsilon>0} d_\tau(b_\varepsilon).
\]
Thus, we can choose $\varepsilon>0$ such that
\[
t< d_\tau(a_\varepsilon), d_\tau(b_\varepsilon).
\]
We have $[a_{\varepsilon/2}]\ll[a]$ and $[b_{\varepsilon/2}]\ll[b]$ in $\Cu(A)$, which allows us to choose $n$ such that $[a_{\varepsilon/2}],[b_{\varepsilon/2}]\leq n[1]$.
Note that $n[1]$ is the class of $1\otimes 1_n$, the unit of the hereditary subalgebra $A\otimes M_n$ in $A\otimes\KK$.
We have
\[
a_\varepsilon = (a_{\varepsilon/2}-\tfrac{\varepsilon}{2})_+,\andSep
a_{\varepsilon/2} \precsim 1\otimes 1_n.
\]
Applying R{\o}rdam's Lemma, we can choose a positive element $a'$ in the hereditary sub-\ca{} generated by $1\otimes 1_n$ (that is, in $(A\otimes M_n)_+$) such that $a_\varepsilon\sim a'$.
Similarly, we obtain $b'\in (A\otimes M_n)_+$ such that $b_\varepsilon\sim b'$.

The restriction of $\tfrac{1}{n}\tau_\infty$ to $A\otimes M_n$ is an extreme, normalized $2$-quasitrace of $A\otimes M_n$.
Applying \autoref{prp:EdwardsCaPre}, we obtain $c\in (A\otimes M_n)_+$ such that $t< d_\tau(c)$ and $c\precsim a',b'$.
Considering $c$ as an element in $A\otimes\KK$, we have $c\precsim a,b$.
This shows that $c$ has the desired properties.
\end{proof}

\section{Realizing chisels}
\label{sec:chisel}

Let $K$ be a compact, convex set.
In this section, we introduce \emph{chisels}, which are defined as functions in $\LAff(K)$ that take the value $\infty$ except at one extreme point of $K$;
see \autoref{dfn:chisel}.
We show that Edwards' condition implies that strictly positive chisels are realized;
see \autoref{prp:realizeChisel}.
We start with two preparatory lemmas.

\begin{lma}
\label{prp:fctlSimple}
Let $S$ be a simple, non-elementary \CuSgp{} satisfying \axiomO{5} and \axiomO{6}, and let $\lambda\in F(S)$.
Then $\lambda(S)=\{0,\infty\}$ or $\lambda(S)=[0,\infty]$.
\end{lma}
\begin{proof}
Let $a\in S$ such that $\lambda(a)\neq 0,\infty$.
We need to verify that $\lambda(S)=[0,\infty]$.
It follows from Robert's Glimm Halving for simple \CuSgp{s}, \cite[Proposition~5.2.1]{Rob13Cone}, that for every $n\in\NN$ there exists $x\in S$ with $x\neq 0$ and $nx\leq a$.
It follows that for every $n\geq 1$ there exists $y_n\in S$ with $0<\lambda(y_n)<\tfrac{1}{n}$.
Let $t\in(0,\infty)$.
We may choose $k_n\in\NN$ such that $t=\sum_n k_n \lambda(y_n)$.
The sequence $(\sum_{n=1}^l k_n y_n)_l$ is increasing.
Its supremum $y$ satisfies
\[
\lambda(y)
=\lambda\left(\sup_l \sum_{n=1}^l k_n y_n\right)
=\sup_l \lambda\left(\sum_{n=1}^l k_n y_n\right)
=\sup_l \sum_{n=1}^l k_n \lambda( y_n)
= t,
\]
which shows that $t\in\lambda(S)$.
\end{proof}

\begin{lma}
\label{prp:fctlIndep}
Let $S$ be a simple, non-elementary \CuSgp{} satisfying \axiomO{5} and \axiomO{6}, let $u\in S$ be a compact, full element, set $K:=F_{u\mapsto 1}(S)$, and let $\lambda\in\partial_e K$ and $\mu\in K$ with $\lambda\neq\mu$.
Assume that $S$ satisfies Edwards' condition for $\lambda$.
Then for every $n\geq 1$ there exists $a\in S$ with $\lambda(a)<\tfrac{1}{n}$ and $\mu(a)\geq n$.
\end{lma}
\begin{proof}
We first show that $\lambda\nleq\mu$.
By \cite[Proposition~2.2.3]{Rob13Cone}, the cone $F(S)$ is algebraically ordered.
Thus, assuming $\lambda\leq\mu$, there exists $\varphi\in F(S)$ such that $\lambda+\varphi=\mu$.
Since $\lambda(u)=\mu(u)=1$, we obtain $\varphi(u)=0$ and consequently $\varphi=0$.
But then $\lambda=\nu$, a contradiction.

Thus, we have $\lambda\nleq\mu$, which allows us to choose $x\in S$ with $\lambda(x)>\mu(x)$.
Choose $x'\in S$ with $x'\ll x$ and $\lambda(x')>\mu(x')$.
Let $n\geq 1$.
Choose $k,m\in\NN$ such that
\[
k\leq m\lambda(x'), \andSep m\mu(x')\leq (k-n).
\]
Then
\[
k-\tfrac{1}{n} < \lambda(ku), \lambda(mx').
\]
Using that $S$ satisfies Edwards' condition for $\lambda$, we obtain $y\in S$ such that
\[
k-\tfrac{1}{n} < \lambda(y), \andSep y\leq ku, mx'.
\]
Choose $y'\in S$ with
\[
y'\ll y, \andSep k-\tfrac{1}{n} < \lambda(y').
\]
Applying \axiomO{5} for $y'\ll y\leq ku$, we obtain $a\in S$ such that
\[
y'+a\leq ku \leq y+a.
\]
Then
\[
k-\tfrac{1}{n}+\lambda(a)< \lambda(y')+\lambda(a) \leq \lambda(ku)=k,
\]
which implies $\lambda(a)<\tfrac{1}{n}$.
Using that $y\leq mx'$, we deduce
\[
k = \mu(ku) \leq \mu(y)+\mu(a) \leq \mu(mx')+\mu(a) \leq k-n+\mu(a),
\]
which implies that $n\leq \mu(a)$.
\end{proof}

Let $K$ be a compact, convex set, and let $\lambda\in\partial_e K$.
Set $U:=K\setminus\{\lambda\}$, which is an open, convex set.
Therefore the characteristic function of $U$ is lower semicontinuous and concave.
Multiplying by $\infty$, the function becomes affine, which justifies the following:

\begin{dfn}
\label{dfn:chisel}
Let $K$ be a compact, convex set, and let $\lambda\in\partial_e K$.
We let $\sigma_\lambda\in\LAff(K)$ be given by $\sigma_\lambda(\lambda)=0$ and $\sigma_\lambda(\mu)=\infty$ for $\mu\neq\lambda$.
Given $t\in\RR$, the \emph{chisel} at $\lambda$ with value $t$ is the function $t+\sigma_\lambda\in\LAff(K)$.
\end{dfn}

\begin{prp}
\label{prp:realizeChisel}
Let $S$ be a countably based, simple, non-elementary \CuSgp{} satisfying \axiomO{5} and \axiomO{6}, let $u\in S$ be a compact, full element, and set $K:=F_{u\mapsto 1}(S)$.
Assume that $S$ satisfies Edwards' condition for $\partial_e K$.

Then for every $\lambda\in\partial_e K$ and $t>0$ there exists $a\in S$ with $\widehat{a}_{|K}=t+\sigma_\lambda$.
\end{prp}
\begin{proof}
Let $\lambda\in\partial_e K$ and $t>0$.
Consider the set
\[
L := \big\{ a : \lambda(a)<t \big\}.
\]
It follows from \autoref{prp:EdwardsDual} that $L$ is upward directed.
Since $S$ is countably based, every upward directed subset of $S$ has a supremum;
see \autoref{prp:ctblBasedDCPO}.
Thus, we may define
\[
a := \sup L.
\]
By \autoref{prp:fctlSimple}, we have $\lambda(S)=[0,\infty]$, which implies $\lambda(a)=t$.
Given $\mu\in K$ with $\mu\neq\lambda$, it follows from \autoref{prp:fctlIndep} that $\mu(a)=\infty$.
Hence $\widehat{a}_{|K}=t+\sigma_\lambda$.
\end{proof}

\begin{thm}
\label{prp:realizeChiselCa}
Let $A$ be a separable, unital, simple, non-elementary, stably finite \ca{}.
Set $K:=\QT_{1\mapsto 1}(A)$.
Then for every $\tau\in\partial_e K$ and $t>0$ there exists $a\in(A\otimes\KK)_+$ such that $d_\tau(a)=t$ and $d_\mu(a)=\infty$ for every $\mu\in K\setminus\{\tau\}$.
\end{thm}
\begin{proof}
Let $\tau\in\partial_e K$ and $t>0$.
Set $S:=\Cu(A)$.
It follows from the properties of $A$ that $S$ is a countably based, simple, non-elementary, stably finite \CuSgp{} satisfying \axiomO{5} and \axiomO{6}.
The class $u:=[1_A]$ is a compact, full element in $S$.
Under the identification of $\QT(A)$ with $F(S)$, see \autoref{pgr:qt}, the set $K$ corresponds to $F_{u\mapsto 1}(S)$, and $\tau$ corresponds to $d_\tau$.

It follows from \autoref{prp:EdwardsCa} that $S$ satisfies Edwards' condition for $\partial_e K$.
We may therefore apply \autoref{prp:realizeChisel} for $S$ to obtain $s\in S$ such that $\widehat{s}_{|K}=t+\sigma_{d_\tau}$.
Any $a\in(A\otimes\KK)_+$ with $s=[a]$ has the desired properties.
\end{proof}

\section{A new Riesz decomposition axiom for Cuntz semigroups}
\label{sec:O6+}

In this section, we introduce a new axiom \axiomO{6+} for \CuSgp{s}.
We show that this axiom is satisfied by Cuntz semigroups of \ca{s} with stable rank one (\autoref{prp:RieszSR1}), although it is not satisfied by the Cuntz semigroups of all \ca{s};
see \autoref{exa:notO6}.
We also show that \axiomO{6+} does not simply follow from weak cancellation and \axiomO{6} (which are both known to hold for Cuntz semigroups of \ca{s} with stable rank one);
see \autoref{exa:cancNotImplyO6}.

\begin{dfn}
\label{dfn:O6+}
Let $S$ be a \CuSgp.
We say that $S$ satisfies \emph{strengthened almost Riesz decomposition}, or that $S$ satisfies \axiomO{6+}, if for every $a,b,c,x',x,y',y\in S$ satisfying
\[
a \leq b+c, \quad x'\ll x \leq a,b, \andSep y'\ll y \leq a,c,
\]
there exist $e,f\in S$ such that
\[
a \leq e+f, \quad x'\ll e\leq a,b, \andSep y'\ll f\leq a,c.
\]
\end{dfn}

\begin{rmk}
Axiom \axiomO{6+} is a strengthening of the axiom of `almost Riesz decomposition', also called \axiomO{6}, as introduced by Robert in \cite{Rob13Cone}.
It is known that \axiomO{6} holds for the Cuntz semigroup of every \ca{};
see \cite[Proposition~5.1.1]{Rob13Cone}.
The strengthening of \axiomO{6} to obtain \axiomO{6+} is similar to the step from the original formulation of the axiom of `almost algebraic order' to its strengthened version as introduced in \cite[Definition~4.1]{AntPerThi18:TensorProdCu}.
All these axioms state that a certain problem (to find a complement, or to decompose an element below a sum) has an approximate solution, but the strengthened versions state that the element(s) can be found to improve a previously known (partial) solution of the problem.

Note that the strengthened version of \axiomO{5} holds for the Cuntz semigroup of every \ca{};
see \cite[Theorem~4.7]{AntPerThi18:TensorProdCu}.
However, in \autoref{exa:notO6} we show that \axiomO{6+} does not hold for Cuntz semigroups of all \ca{s}.
\end{rmk}

\begin{lma}
\label{prp:preO6+}
Let $S$ be a \CuSgp.
Assume that for every $a',a,b,c,x',x\in S$ satisfying
\[
a'\ll a \leq b+c, \andSep x'\ll x \leq a,b,
\]
there exists $e\in S$ such that
\[
a' \leq e+c, \andSep x'\ll e\leq a,b.
\]
Then $S$ satisfies \axiomO{6+}.
\end{lma}
\begin{proof}
Claim~1:
Let $a,b,c,x',x\in S$ satisfy
\[
a \leq b+c, \andSep x'\ll x \leq a,b.
\]
Then there exists $e\in S$ such that
\[
a \leq e+c, \andSep x'\ll e\leq a,b.
\]

To verify the claim, set $e_0':=x'$, and choose $e_0$ such that $e_0'\ll e_0\ll x$.
Then choose a $\ll$-increasing sequence $(a_n)_{n\geq 0}$ with supremum $a$ and such that $a_0=e_0$.
For each $n\geq 1$, choose $a_n'$ satisfying $a_{n-1}\ll a_n'\ll a_n$.
We will inductively find elements $e_n',e_n$, for $n\geq 1$, such that
\[
a_n' \leq e_n'+c, \andSep e_{n-1}'\ll e_n' \ll e_n \leq a_{n+1},b,
\]
for $n\geq 1$. 
For the base case of the induction, we have $e_0'=x'$ and $e_0=a_0$.

For the induction step, let $n\geq 1$, and assume that we have chosen $e_k'$ and $e_k$ for $k\leq {n-1}$.
To obtain $e_n'$ and $e_n$, we use that
\[
a_n \ll a_{n+1} \leq b+c, \andSep e_{n-1}'\ll e_{n-1} \leq a_{n+1},b.
\]
By assumption, we obtain $e_n\in S$ such that
\[
a_n \leq e_n + c, \andSep e_{n-1}'\ll e_n\leq a_{n+1},b.
\]
Using that $a_n'\ll a_n$ and $e_{n-1}'\ll e_n$, we can choose $e_n'$ such that
\[
a_n' \leq e_n' + c, \andSep e_{n-1}'\ll e_n'\ll e_n.
\]
Then $e_n'$ and $e_n$ have the desired properties.

Using that $(e_n')_n$ is an increasing sequence, we may set $e:=\sup_n e_n'$.
For each $n$ we have $a_{n-1}\leq a_n' \leq e_n'+c$, and therefore $a\leq e+c$.
Further, we have $x'=e_0'\ll e_1' \leq e$.
Moreover, for each $n$ we have $e_n'\leq a_{n+2}\leq a$ and $e_n'\leq b$, and thus $e\leq a,b$.
This proves the claim.

To verify that $S$ satisfies \axiomO{6+}, let $a,b,c,x',x,y',y\in S$ satisfy
\[
a \leq b+c, \andSep x'\ll x \leq a,b, \andSep y'\ll y \leq a,c.
\]
Applying claim~1 for $a \leq b+c$ and $x'\ll x \leq a,b$, we obtain $e\in S$ such that
\[
a \leq e+c, \andSep x'\ll e\leq a,b.
\]
Applying claim~1 for $a \leq e+c$ and $y'\ll y \leq a,c$, we obtain $f\in S$ with
\[
a \leq e+f, \andSep y'\ll f\leq a,c.
\]
It follows that $e$ and $f$ have the desired properties.
\end{proof}

Recall that a unital \ca{} is said to have \emph{stable rank one} if its invertible elements are dense.
A nonunital \ca{} is defined to have stable rank one if its minimal unitization does.
Stable rank one passes to quotients, hereditary sub-\ca{s}, matrix algebras and stabilizations.

\begin{thm}
\label{prp:RieszSR1}
Let $A$ be a \ca{} that has stable rank one.
Then $\Cu(A)$ satisfies \axiomO{6+}.
\end{thm}
\begin{proof}
We use the notation for Cuntz (sub)equivalence and $\varepsilon$-cut-downs as in \autoref{pgr:CuA}.
Given $\delta>0$, let $f_\delta\colon[0,\infty)\to[0,1]$ be the continuous function that takes value $0$ at $0$, that takes value $1$ on $[\delta,\infty)$, and that is linear on the interval $[0,\delta]$.
Given $z\in A_+$, the element $f_\delta(z)$ acts as a unit on $z_\delta$, and consequently $f_\delta(z)d=d=df_\delta(z)$ for every $d\in\overline{z_\delta A z_\delta}$.

Since stable rank one passes to stabilizations, we may assume that $A$ is stable.
Set $S:=\Cu(A)$.
To verify the assumption of \autoref{prp:preO6+} for $S$, let $a',a,b,c,x',x\in S$ satisfy
\[
a'\ll a \leq b+c, \andSep x'\ll x \leq a,b.
\]
Choose $\tilde{x}\in S$ such that $x'\ll\tilde{x}\ll x$.
Choose $r,s,t\in A_+$ representing $a$, $b$ and $c$, respectively.
We may assume that $s$ and $t$ are orthogonal.
Then $b+c=[s+t]$.
Choose $\varepsilon>0$ such that $\tilde{x},a'\leq[r_{\varepsilon}]$.
Since $\tilde{x}\ll x$, we can choose $\delta'>0$ such that $\tilde{x}\leq[s_{\delta'}]$.

Since $r\precsim s+t$, there exists $\delta>0$ and $\tilde{r}\in A_+$ such that $r_{\varepsilon}\sim\tilde{r}\subseteq(s+t)_\delta$.
We may assume that $\delta\leq\delta'$, so that also $\tilde{x}\leq[s_{\delta}]$.
Choose $g\in A_+$ representing $\tilde{x}$.
Take $\alpha>0$ such that $x'\ll[g_\alpha]$.
We have
\[
[g] = \tilde{x} \leq [r_\varepsilon] = [\tilde{r}].
\]
Thus, there exist $\beta>0$ and $\tilde{g}\in A_+$ such that $g_\alpha\sim\tilde{g}\subseteq\tilde{r}_\beta$.
Then $f_\beta(\tilde{r})$ acts as a unit on $\tilde{g}$ and therefore
\[
\tilde{g} = f_\beta(\tilde{r}) \tilde{g} f_\beta(\tilde{r}) \leq \|\tilde{g}\| f_\beta(\tilde{r})^2 \leq \|\tilde{g}\| \beta^{-1} \tilde{r}.
\]
We have $x'\ll[g_\alpha]=[\tilde{g}]$.
Choose $\eta>0$ such that $x'\ll[\tilde{g}_\eta]$.

Set $B:=\overline{(s+t)_\delta A(s+t)_\delta}$.
We have
\[
\tilde{g} \sim g_\alpha \precsim g \precsim s_\delta,
\]
in $A$.
Since both $\tilde{g}$ and $s_\delta$ belong to $B$, we deduce that $\tilde{g} \precsim s_\delta$ in $B$.
Since stable rank one passes to hereditary sub-\ca{s}, $B$ has stable rank one.
Thus, there exists a unitary $u$ in the minimal unitization of $B$ such that $u \tilde{g}_\eta u^*\subseteq s_\delta$;
see \cite[Proposition~2.4]{Ror92StructureUHF2}.
Set
\[
\hat{g} := u \tilde{g}_\eta u^*, \andSep \hat{r}:=u\tilde{r}u^*.
\]
Then $\tilde{r}\sim\hat{r}\subseteq(s+t)_\delta$.
Set
\[
e:=[\hat{r} f_\delta(s) \hat{r}], \andSep f:=[\hat{r} f_\delta(t) \hat{r}].
\]
We have
\[
\hat{r} f_\delta(s) \hat{r} \sim f_\delta(s)^{1/2} \hat{r}^2 f_\delta(s)^{1/2} \precsim \hat{r}^2 \sim \hat{r},\andSep
\hat{r} f_\delta(s) \hat{r} \precsim f_\delta(s) \precsim s,
\]
and therefore $e\leq a$ and $e\leq b$.
Similarly, we obtain $f\leq a,c$.

Since $s$ and $t$ are orthogonal, we have $f_\delta(s+t)=f_\delta(s)+f_\delta(t)$.
Using this at the fifth step, and using at the fourth step that $f_\delta(s+t)$ acts as a unit on $\hat{r}$, we deduce
\begin{align*}
a'
\leq[r_\varepsilon]
=[\hat{r}]
=[\hat{r}^2]
&=[\hat{r} f_\delta(s+t) \hat{r}] \\
&=[\hat{r} f_\delta(s) \hat{r} + \hat{r} f_\delta(t) \hat{r}] \\
&\leq [\hat{r} f_\delta(s) \hat{r}] + [\hat{r} f_\delta(t) \hat{r}] = e+f \leq e+c.
\end{align*}
For $\kappa:=\|\tilde{g}\| \beta^{-1}$ we have $\tilde{g} \leq \kappa \tilde{r}$ and consequently
\[
\hat{g} = u \tilde{g}_\eta u^*
\leq u \tilde{g} u^*
\leq \kappa u \tilde{r} u^*
= \kappa \hat{r}.
\]
Then
\[
\hat{g}
= f_\delta(s)^{1/2}\hat{g}f_\delta(s)^{1/2}
\leq \kappa f_\delta(s)^{1/2}\hat{r}f_\delta(s)^{1/2}
\sim \hat{r}f_\delta(s)\hat{r},
\]
which implies
\[
x' \ll [\tilde{g}_\eta] = [u \tilde{g}_\eta u^*] = [\hat{g}] \leq [\hat{r}f_\delta(s)\hat{r}] = e,
\]
as desired.
\end{proof}

\begin{rmk}
\label{rmk:RieszSR1}
\autoref{prp:RieszSR1} also holds if $A\otimes\KK$ is only assumed to `have almost stable rank one'.
Recall that a \ca{} $A$ is said to \emph{have almost stable rank one} if every hereditary sub-\ca{} of $A$ is contained in the closure of the invertible elements of its minimal unitization;
see \cite[Definition~3.1]{Rob16RmksZstblProjless}.
In the proof of \autoref{prp:RieszSR1}, the assumption of stable rank one was only used to obtain the unitary $u$ in the minimal unitization of $B$.
The existence of such a unitary follows also if $B$ is contained in the closure of the invertible elements of its minimal unitization.

If $A$ has stable rank one, then it has almost stable rank one.
The converse holds if $A$ is unital.
It was shown by Robert, \cite[Corollary~3.2]{Rob16RmksZstblProjless}, that every (even not necessarily simple) $\mathcal{Z}$-stable, projectionless \ca{} has almost stable rank one.
It is not known if such algebras actually have stable rank one.

While stable rank one passes from a \ca{} to its stabilization, the same does not hold for almost stable rank one;
see \autoref{exa:asr1NotStable}.
Thus, to obtain \axiomO{6+} for the Cuntz semigroup of a \ca{} $A$, we need to require almost stable rank for $A\otimes\KK$ and not just for $A$.
\end{rmk}

We thank Leonel Robert for providing the following example.

\begin{exa}[Robert]
\label{exa:asr1NotStable}
Almost stable rank one does not pass to matrix algebras.
The Cuntz semigroup of the Jiang-Su algebra $\mathcal{Z}$ is isomorphic to $\NN\sqcup(0,\infty]$, with the elements in $\NN$ being compact (given by classes of projections), and with elements in $(0,\infty]$ being soft.
Choose a positive element $h$ in $\mathcal{Z}$ representing the soft $1$, and set $A:=\overline{h\mathcal{Z}h}$, the hereditary sub-\ca{} of $\mathcal{Z}$ generated by $h$.
Then $A$ is $\mathcal{Z}$-stable, non-unital, and projectionless.
However, $A\otimes M_2$ contains the class of the unit of $\mathcal{Z}$.
Hence, $\mathcal{Z}$ is isomorphic to a hereditary sub-\ca{} of $A$.
In particular, $A$ is projectionless, but not stably projectionless.

Let $D:=\{z\in\CC : |z|\leq 1\}$ be the $2$-disc, and consider $B:=C(D,A)$.
This \ca{} is again $\mathcal{Z}$-stable and projectionless.
Therefore, it has almost stable rank one, by \cite[Corollary~3.2]{Rob16RmksZstblProjless}.
However, $B\otimes M_2$ contains $C(D,\mathcal{Z})$ as a hereditary sub-\ca.
Let $f\colon D\to\mathcal{Z}$ be given by $f(z)=z\cdot 1_{\mathcal{Z}}$. 
We claim that $f$ can not be approximated by invertible elements in $C(D,\mathcal{Z})$.
Identify the boundary of $D$ with the 1-sphere $S^1$.
Then $f_{|S^1}$ defines a nontrivial element in $K_1(C(S^1)\otimes\mathcal{Z}) \cong K_1(C(S^1)) \cong \ZZ$.
Let $g\in C(D,\mathcal{Z})$ be invertible.
Then $g_{|S^1}$ defines the trivial element in $K_1(C(S^1)\otimes\mathcal{Z})$.
However, if $\|f-g\|_\infty<1$, then $f_{|S^1}$ and $g_{|S^1}$ define the same element in $K_1(C(S^1)\otimes\mathcal{Z})$.
Hence, $\|f-g\|_\infty\geq 1$.


Therefore $C(D,\mathcal{Z})$ does not have stable rank one.
Since $C(D,\mathcal{Z})$ is a unital, hereditary sub-\ca{} of $B\otimes M_2$, it follows that $B\otimes M_2$ does not have almost stable rank one.
\end{exa}

In \autoref{prp:LAffCu}, we have observed that for certain \CuSgp{s}, \axiomO{6} is equivalent to \axiomO{6+}.
However, while the Cuntz semigroup of every \ca{} satisfies \axiomO{6}, the next example shows that the same does not hold for \axiomO{6+}.

\begin{exa}
\label{exa:notO6}
Axiom \axiomO{6+} does not hold for the Cuntz semigroups of all \ca{s}.
Consider $C(S^2)$, the \ca{} of continuous functions on the $2$-sphere.
Nonzero vector bundles over $S^2$ are classified by their ranks (taking values in $\NN_{>0}$) together with their first Chern classes (taking values in $\ZZ$).
It follows that the Murray-von Neumann semigroup of $C(S^2)$ is
\[
V(C(S^2)) \cong \big\{ (0,0) \big\} \cup \big\{ (n,m) : n>0 \big\} \subseteq \ZZ\times\ZZ,
\]
with the algebraic order.

Let $\Lsc(S^2,\NNbar)$ denote the set of lower semicontinuous functions $S^2\to\NNbar$ with pointwise order and addition.
Given an open subset $U\subseteq S^2$, we let $\charFct_U\in\Lsc(S^2,\NNbar)$ denote the characteristic function of $U$.
We set
\[
\Lsc(S^2,\NNbar)_{nc} := \Lsc(S^2,\NNbar)\setminus\{n\charFct_{S^2}:n\geq 1\},
\]
the set of nonconstant functions in $\Lsc(S^2,\NNbar)$ together with zero.
By \cite[Theorem~1.2]{Rob13CuSpDim2}, we have
\begin{align*}
\Cu(C(S^2))
&\cong V(C(S^2)) \sqcup \big(\Lsc(S^2,\NNbar)\setminus\{n\charFct_{S^2}:n\geq 1\} \big) \\
&\cong \left( \NN_{>0}\times\ZZ \right) \sqcup \Lsc(S^2,\NNbar)_{nc}.
\end{align*}
Let us describe the order and addition.
We retain the usual addition and order (algebraic and pointwise, respectively) in the two components $\NN_{>0}\times\ZZ$ and $\Lsc(S^2,\NNbar)_{nc}$.
Let $(n,m)\in\NN_{>0}\times\ZZ$ and let $f\in\Lsc(S^2,\NNbar)_{nc}$.
We set $(n,m)+f:=n\charFct_{S^2}+f$ in $\Lsc(S^2,\NNbar)_{nc}$.
We have $(n,m)\leq f$ if and only if $n\charFct_{S^2}\leq f$ in $\Lsc(S^2,\NNbar)$.
Similarly, $f\leq (n,m)$ if and only if $f\leq n\charFct_{S^2}$.

Let $U,V\subseteq S^2$ be proper, nonempty, open subsets such that the closure of $U$ is contained in $V$.
Then $\charFct_U\ll \charFct_V$ in $\Cu(C(S^2))$, and $\charFct_V,\charFct_U\in\Lsc(S^2,\NNbar)_{nc}$.
We have
\[
(1,0) \leq (1,1) + \charFct_U, \andSep \charFct_U\ll \charFct_V \leq (1,0),(1,1).
\]
If $\Cu(C(S^2))$ satisfied \axiomO{6+}, then there would be $e\in\Cu(C(S^2))$ such that
\[
(1,0) \leq e + \charFct_U, \andSep \charFct_U\leq e \leq (1,0),(1,1).
\]
The condition $e\leq(1,0),(1,1)$ implies that $e$ belongs to $\Lsc(S^2,\NNbar)_{nc}$.
Thus, there exists $t\in S^2$ with $e(t)=0$.
Since $\charFct_U\leq e$, we have $t\notin U$.
But then $(e + \charFct_U)(t)=0$, which implies that $(1,0)$ is not dominated by $e+\charFct_U$, a contradiction.
\end{exa}

\begin{exa}
\label{exa:cancNotImplyO6}
Axiom \axiomO{6+} does not follow from weak cancellation and \axiomO{6}.
For example, consider again $S:=\Cu(C(S^2))$.
Since $S$ is the Cuntz semigroup of a \ca{}, it satisfies \axiomO{6} (and also \axiomO{5}).
Using the description of $S$ from \autoref{exa:notO6}, it is straightforward to check that $S$ is weakly cancellative.
However, by \autoref{exa:notO6}, $S$ does not satisfy \axiomO{6+}.
\end{exa}

\begin{rmk}
\label{rmk:RieszDecomp}
The Cuntz semigroup of a \ca{} with stable rank one need not satisfy Riesz decomposition.
Consider for example the Jiang-Su algebra $\mathcal{Z}$.
We have $\Cu(\mathcal{Z})\cong\NN\sqcup(0,\infty]$.
The elements in $\NN$ are compact, while the elements in $(0,\infty]$ are soft.

Let $a=1$, the compact element of value one, and let $b=c=\tfrac{2}{3}$, the soft element of value $\tfrac{2}{3}$.
We have $a=1\leq\tfrac{4}{3}=b+c$.
The only compact elements below $1$ are $0$ and $1$.
Moreover, the sum of any element with a nonzero soft element is again soft.
Therefore, if $a=a_1+a_2$, then either $a_1=0$ and $a_2=1$, or reversed.
In either case, we do not have $a_1,a_2\leq \tfrac{2}{3}$.

On the positive side, Perera has shown that the Cuntz semigroup of a $\sigma$-unital \ca{} with stable rank one and real rank zero satisfies Riesz decomposition;
see \cite[Theorem~2.13]{Per97StructurePositive}.
\end{rmk}

\section{Realizing functional infima}
\label{sec:fctlInf}

Let $S$ be a simple, stably finite \CuSgp{} satisfying \axiomO{5} and \axiomO{6}, let $u\in S$ be a compact, full element, and set $K:=F_{u\mapsto 1}(S)$.
Then $\LAff(K)$ is an inf-semilattice by \autoref{prp:FSChoquet} and \autoref{prp:LAffRiesz}.
We will consider the property that the set $\{\widehat{a}_{|K}:a\in S\}$ is closed under infima:

\begin{dfn}
\label{dfn:fctlInf}
Let $S$ be a simple, stably finite \CuSgp{} satisfying \axiomO{5} and \axiomO{6}, let $u\in S$ be a compact, full element, and set $K:=F_{u\mapsto 1}(S)$.
We say that $S$ \emph{realizes functional infima over $K$} if for all $a,b\in S$ there exists $c\in S$ satisfying $\widehat{c}_{|K}=\widehat{a}_{|K}\wedge\widehat{b}_{|K}$.
\end{dfn}

In \autoref{prp:fctlInf}, we show that \axiomO{6+} together with Edwards' condition ensures that $S$ realizes functional infima.
We prepare the proof in a sequence of lemmas.
In the following result, note that the backward implications in~(1) and~(2) follow immediately from $c\leq a,b$.

\begin{lma}
\label{prp:prePseudoInf}
Let $S$ be a \CuSgp{} satisfying \axiomO{6+}, let $B\subseteq S$ be a countable set, and let $z',z,a,b\in S$ satisfy $z'\ll z\leq a,b$.
Then there exists $c\in S$ such that $z'\ll c\leq a,b$ and:
\begin{enumerate}
\item
For every $t\in B$, we have $a\leq b+t$ if and only if $a\leq c+t$.
\item
For every $t\in B$, we have $b\leq a+t$ if and only if $b\leq c+t$.
\end{enumerate}
\end{lma}
\begin{proof}
Let $B_a$ be the set of elements $t\in B$ such that $a\leq b+t$, and similarly for $B_b$.
Choose a sequence $(t_n)_n$ of elements in $B_a$ such that every element of $B_a$ appears cofinally often.
Similarly, choose the sequence $(r_n)_n$ of elements in $B_b$.

Choose $\ll$-increasing sequences $(a_n)_n$ and $(b_n)_n$ with supremum $a$ and $b$, respectively.
We will inductively choose elements $x_n',x_n,y_n',y_n\in S$ satisfying
\[
a_n \leq x_n' + t_n, \andSep y_{n-1}'\ll x_n'\ll x_n \leq a,b,
\]
and
\[
b_n \leq y_n' + r_n, \andSep x_n'\ll y_n'\ll y_n \leq a,b.
\]

Set $x_0',x_0:=0$, $y_0':=z'$ and $y_0:=z$.
Assume that we have chose $x_k',x_k,y_k',y_k$ for $k\leq n-1$.
We have $a\leq b+t_n$ and $y_{n-1}'\ll y_{n-1}\leq a,b$.
Applying \axiomO{6+}, we obtain $x_n$ such that
\[
a \leq x_n + t_n, \andSep y_{n-1}'\ll x_n \leq  a,b.
\]
Using that $a_n\ll a$, we can choose $x_n'$ such that
\[
a_n \leq x_n' + t_n, \andSep y_{n-1}'\ll x_n' \ll x_n.
\]
We have $b\leq a+r_n$ and $x_n'\ll x_n\leq a,b$.
Applying \axiomO{6+}, we obtain $y_n$ such that
\[
b \leq y_n + r_n, \andSep x_n'\ll y_n \leq a,b.
\]
Using that $b_n\ll b$, we can choose $y_n'$ such that
\[
b_n \leq y_n' + r_n, \andSep x_n'\ll y_n' \ll y_n.
\]

Set $c:=\sup_n x_n'$.
Note that $c=\sup_n y_n'$.
For each $n$, we have $x_n'\leq a,b$, and therefore $c\leq a,b$.
Moreover, we have $z'=y_0'\ll c$.

Let $t\in B_a$. 
By the choice of the sequence $(t_n)_n$, there exists a strictly increasing sequence $(n_k)_k$ of natural numbers such that $t=t_{n_k}$ for each $k$.
We have
\[
a_{n_k} \leq x_{n_k}' + t_{n_k} = x_{n_k}' + t \leq c+t,
\]
for each $k\in\NN$, and therefore $a\leq c+t$.

Similarly, we obtain $b\leq c+t$ for every $t\in B_b$. 
\end{proof}

\begin{lma}
\label{prp:pseudoInf}
Let $S$ be countably based \CuSgp{} satisfying \axiomO{5} and \axiomO{6+}, and let $z',z,a,b\in S$ satisfy $z'\ll z\leq a,b$.
Then there exists $c\in S$ satisfying $z'\ll c\leq a,b$ and:
\begin{enumerate}
\item
Given $y'\ll y\leq a,b$, there exists $r\in S$ such that $y'+r \leq a \leq c+r$.
\item
Given $y'\ll y\leq a,b$, there exists $s\in S$ such that $y'+s \leq b \leq c+s$.
\end{enumerate}
\end{lma}
\begin{proof}
Using that $S$ is countably based we can choose $x_n',x_n\in S$, for $n\in\NN$, such that $x_n'\ll x_n$ for each $n$, and such that for any $y'\ll y\leq a,b$ there exists $n$ such that $y'\leq x_n'\ll x_n \leq y$.
Given $n$, apply \axiomO{5} for $x_n'\ll x_n\leq a$ to obtain $r_n$ such that
\[
x_n' + r_n \leq a \leq x_n + r_n.
\]
Similarly, we obtain $s_n$ such that
\[
x_n' + s_n \leq b \leq x_n + s_n.
\]
Let $B$ be the set consisting of the elements $r_n$ and $s_n$.

Apply \autoref{prp:prePseudoInf} for $z',z,a,b$ and $B$ to obtain $c\in S$ satisfying $z'\ll c\leq a,b$ and the statements of \autoref{prp:prePseudoInf}.
Let $y',y\in S$ satisfy $y'\ll y\leq a,b$.
By construction, we can choose $n$ such that $y'\leq x_n'\ll x_n \leq y$.
Then
\[
a \leq x_n + r_n \leq b + r_n.
\]
Using \autoref{prp:prePseudoInf}(1), we obtain $a\leq c+r_n$.
Then
\[
y'+r_n \leq x_n'+r_n \leq a \leq c+r_n.
\]
Similarly, one deduces the second statement from \autoref{prp:prePseudoInf}(2).
\end{proof}

\begin{lma}
\label{prp:fctlInfPre}
Let $S$ be countably based, simple, stably finite \CuSgp{} satisfying \axiomO{5} and \axiomO{6+}, let $u\in S$ be a compact, full element, let $z',z,a,b\in S$ satisfy $z'\ll z\leq a,b$, and set $K:=F_{u\mapsto 1}(S)$.
Then there exists $c\in S$ satisfying $z'\ll c\leq a,b$ and such that
\[
\sup\big\{ \lambda(x) : x\leq a,b \big\}
\leq \lambda(c),
\]
for every $\lambda\in K$ with $\lambda(a)<\infty$. 
\end{lma}
\begin{proof}
Apply \autoref{prp:pseudoInf} for $z',z,a,b$ to obtain $c\in S$ satisfying $z'\ll c\leq a,b$ and the statements of \autoref{prp:pseudoInf}.
To show that $c$ has the desired property, let $\lambda\in K$ satisfy $\lambda(a)<\infty$.
Let $x\in S$ satisfy $x\leq a,b$.
Let $x'\ll c$.
Apply \autoref{prp:pseudoInf}(1) to obtain $r\in S$ such that
\[
x'+r \leq a \leq c+r.
\]
Then
\[
\lambda(x')+\lambda(r) \leq \lambda(c)+\lambda(r).
\]
Since $\lambda(a)<\infty$, we have $\lambda(r)<\infty$, which allows us to cancel it from the above inequalities.
We deduce that
\[
\lambda(x') \leq \lambda(c).
\]
Since $\lambda(x)=\sup\{\lambda(x') : x'\ll x\}$, we obtain $\lambda(x)\leq\lambda(c)$, as desired.
\end{proof}

\begin{thm}
\label{prp:fctlInf}
Let $S$ be countably based, simple, stably finite \CuSgp{} satisfying \axiomO{5} and \axiomO{6+}, let $u\in S$ be a compact, full element, and set $K:=F_{u\mapsto 1}(S)$.
Assume that $S$ satisfies Edwards' condition for $\partial_e K$.
Then for every $a,b\in S$ there exists $c\in S$ satisfying $c\leq a,b$ and $\widehat{c}_{|K}=\widehat{a}_{|K}\wedge\widehat{b}_{|K}$.
\end{thm}
\begin{proof}
By \autoref{prp:FSChoquet}, $K$ is a metrizable Choquet simplex.
Let $a,b\in S$.
Choose a $\ll$-increasing sequence $(a_n)_n$ with supremum $a$.
Set
\[
f:=\widehat{a}_{|K}\wedge\widehat{b}_{|K}, \andSep
f_n:=\widehat{a_n}_{|K}\wedge\widehat{b}_{|K},
\]
for each $n$.
Then $(f_n)_n$ is an increasing sequence in $\LAff(K)$ with pointwise supremum $f$.

By \autoref{prp:llInLAff}, $\LAff(K)$ is a countably based domain.
Therefore, the way-below relation in $\LAff(K)$ agrees with its sequential version, and every element in $\LAff(K)$ is the supremum of a $\ll$-increasing sequence.
We can therefore choose an increasing sequence $(f_n')_n$ in $\LAff(K)$ such that $f_n'\ll f_n$ for each $n$, and such that $\sup_n f_n'=f$.
We may assume $f_0'\leq 0$.

We will inductively construct $c_n',c_n\in S$ for $n\geq 0$ such that
\[
c_{n-1}' \ll c_n'\ll c_n \leq a_n,b,\andSep f_n' \leq \widehat{c_n'}_{|K},
\]
for $n\geq 1$.
Set $c_0'=c_0=0$.
Assume that we have chosen $c_k',c_k$ for $k\leq n-1$.
Apply \autoref{prp:fctlInfPre} for $c_{n-1}'\ll c_{n-1}\leq a_n,b$ to obtain $c_n\in S$ satisfying
\[
c_{n-n}'\ll c_n \leq a_n,b,\andSep \sup\big\{ \lambda(x) : x\leq a_n,b \big\}
\leq \lambda(c_n),
\]
for every $\lambda\in K$ with $\lambda(a_n)<\infty$.
Since $a_n\ll a\leq\infty u$, there exists $N\in\NN$ such that $a_n\leq Nu$.
Every $\lambda\in K$ satisfies $\lambda(u)=1$ and hence $\lambda(a_n)<\infty$.
Given $\lambda\in\partial_e K$, we apply Edwards' condition at the first step to obtain
\[
\min\{\lambda(a_n),\lambda(b)\}
= \sup \big\{ \lambda(x) : x\leq a_n,b \big\}
\leq \lambda(c_n)
\leq \min\{\lambda(a_n),\lambda(b)\}.
\]
It follows that $\widehat{c_n}(\lambda)=\min\{\lambda(a_n),\lambda(b)\}=f_n(\lambda)$, for all $\lambda\in\partial_e K$.
Thus, $\widehat{c_n}_{|K}=f_n$, by \autoref{prp:LAffOrderExtr}.

Use that $f_n'\ll f_n$ and $c_{n-1}'\ll c_n$ to choose $c_n'\in S$ such that
\[
c_{n-n}'\ll c_n' \ll c_n, \andSep  f_n'\leq \widehat{c_n'}_{|K}.
\]

The sequence $(c_n')_n$ is increasing, which allows us to set $c:=\sup_n c_n'$.
We have $c_n'\leq a,b$ for each $n$, and therefore $c\leq a,b$.
Further, we have
\[
f = \sup_n f_n' \leq \sup_n\widehat{c_n'}_{|K}=\widehat{c}_{|K},
\]
which shows that $c$ has the desired properties.
\end{proof}

\begin{rmk}
In \cite{AntPerRobThi18arX:CuntzSR1} we show that \axiomO{6+} together with a new technique of adjoining compact elements can be used to show the existence of infima (not just functional infima).
In particular, it follows that Cuntz semigroups of separable \ca{s} with stable rank one are inf-semilattices.
\end{rmk}

\section{Realizing ranks}
\label{sec:rankCu}

Throughout this section, we fix a weakly cancellative, countably based, simple, stably finite, non-elementary \CuSgp{} $S$ satisfying \axiomO{5} and \axiomO{6+}, together with a compact, full element $u\in S$, and we set $K:=F_{u\mapsto 1}(S)$.
We also assume that $S$ satisfies Edwards' condition for $\partial_e K$.
Note that $K$ is a Choquet simplex.

\begin{lma}
\label{prp:realizeRankApprox}
For every $g\in\Aff(K)_{++}$ and $\varepsilon>0$ there exists $a\in S$ with $g\leq\widehat{a}_{|K}\leq g+\varepsilon$.
\end{lma}
\begin{proof}
Let $g\in\Aff(K)_{++}$ and $\varepsilon>0$.
For each $\lambda\in \partial_e K$, apply \autoref{prp:realizeChisel} to obtain $a_\lambda\in S$ with
\[
\widehat{a_\lambda}_{|K} = g(\lambda)+\tfrac{\varepsilon}{2}+\sigma_{\lambda}.
\]

We may assume that $a_\lambda$ is soft;
see \autoref{rmk:predecessor}.
By \autoref{prp:llInLAff}, we have $g \ll g+\tfrac{\varepsilon}{2} \leq \widehat{a_\lambda}_{|K}$ in $\LAff(K)$, which allows us to choose $a_\lambda'\in S$ such that
\[
a_\lambda'\ll a_\lambda, \andSep g\leq \widehat{a_\lambda'}_{|K}.
\]
By \autoref{prp:wayBelowSoft}, we have $\widehat{a_\lambda'}_{|K}\ll\widehat{a_\lambda}_{|K}$ in $\LAff(K)$.
Use \autoref{prp:llInLAff} to choose $g_\lambda\in\Aff(K)$ satisfying
\[
\widehat{a_\lambda'}_{|K}\leq g_\lambda \leq \widehat{a_\lambda}_{|K}.
\]

Consider the family $\mathcal{F}$ of finite subsets of $\partial_e K$, which is an upward directed set ordered by inclusion.
For each $M\in \mathcal{F}$ set
\[
h_M:=\bigwedge_{\lambda\in M} g_\lambda.
\]

Then $h_M\leq h_N$ if $M,N\in \mathcal{F}$ satisfy $M\supseteq N$.
Thus, $(h_M)_{M\in \mathcal{F}}$ is a downward directed net in $\Aff(K)$.
Set
\[
h:= \lim_{M\in \mathcal{F}} h_M = \inf_{M\in \mathcal{F}} h_M,
\]
the pointwise infimum of the family $(h_M)_{M\in \mathcal{F}}$.
Then $h$ is an upper semicontinuous, affine function on $K$.
For each $\lambda\in\partial_e K$, we have
\[
h(\lambda)
\leq h_{\{\lambda\}}(\lambda)
= g_\lambda(\lambda)
\leq \widehat{a_\lambda}(\lambda)
= g(\lambda)+\tfrac{\varepsilon}{2}.
\]
This implies that $0\leq g(\lambda)+\tfrac{\varepsilon}{2}-h(\lambda)$ for all $\lambda\in\partial_e K$.
Since the function $g+\tfrac{\varepsilon}{2}-h$ belongs to $\LAff(K)$, we may apply \autoref{prp:LAffOrderExtr} to deduce that $0\leq g+\tfrac{\varepsilon}{2}-h$.
Thus, $h\leq g+\tfrac{\varepsilon}{2}$.

By \autoref{prp:llInLAff}, we have $-g-\varepsilon\ll -g-\tfrac{\varepsilon}{2}$.
Hence, the function $g+\varepsilon$ is way-above $g+\tfrac{\varepsilon}{2}$ in the sense that for every downward directed subset $D$ of upper semicontinuous functions with $\inf D\leq g+\tfrac{\varepsilon}{2}$, there exists $d\in D$ with $d\leq g+\varepsilon$.
Thus, we obtain $M\in \mathcal{F}$ such that $h_M\leq g+\varepsilon$.
Applying \autoref{prp:fctlInf}, we obtain $a\in S$ such that
\[
\widehat{a}_{|K} = \bigwedge_{\lambda\in M} \widehat{a_\lambda'}_{|K}.
\]
For each $\lambda\in M$, we have $g\leq \widehat{a_\lambda'}_{|K}$ and therefore $g\leq\widehat{a}_{|K}$.
Then
\[
g \leq \widehat{a}_{|K}
= \bigwedge_{\lambda\in M} \widehat{a_\lambda'}_{|K}
\leq \bigwedge_{\lambda\in M} g_\lambda
= h_M
\leq g+\varepsilon,
\]
which shows that $a$ has the desired properties.
\end{proof}

\begin{rmk}
\label{rmk:realizeRankApprox}
In the proof of \autoref{prp:realizeRankApprox}, the assumption \axiomO{6+} is only used to realize functional infima over $K$.
\end{rmk}

\begin{lma}
\label{prp:LfUpDirected}
For every $f\in\LAff(K)_{++}$, the set
\[
L_f' := \big\{ a'\in S : \exists\, a\in S \text{ with } a'\ll a, \text{ and } \widehat{a}_{|K} \ll f \text{ in } \LAff(K)\big\},
\]
is upward directed.
\end{lma}
\begin{proof}
Let $f\in\LAff(K)_{++}$.
To show that $L_f'$ is upward directed, let $a',a,b',b\in S$ satisfy
\[
a'\ll a,\andSep \widehat{a}_{|K}\ll f, \andSep b'\ll b,\andSep \widehat{b}_{|K}\ll f.
\]
We need to find $c',c\in S$ such that $a',b'\leq c'\ll c$ and $\widehat{c}_{|K}\ll f$.


By \autoref{prp:llInLAff}, $\LAff(K)$ is a domain, which allows us to choose $f'$ such that $\widehat{a}_{|K},\widehat{b}_{|K}\leq f'\ll f$.
Apply \autoref{prp:llInLAff} to choose $g\in\Aff(K)$ and $\varepsilon>0$ such that $f'\leq g$ and $g+2\varepsilon \leq f$. 
Note that $\widehat{a}_{|K},\widehat{b}_{|K}\leq g$ and $g+\varepsilon \ll f$.
We may also assume that $a,b\ll\infty$.
This allows us to choose $n\in\NN$ such that $a,b\leq nu$ and $g\leq \widehat{nu}_{|K}$.
Then there exists $h\in\Aff(K)$ such that $g+h=\widehat{nu}_{|K}$.

Apply \axiomO{5} to obtain $r,s\in S$ such that
\[
a'+r \leq nu \leq a+r,\andSep b'+s \leq nu \leq b+s.
\]
Then
\[
g+h = \widehat{nu}_{|K} \leq \widehat{a}_{|K} + \widehat{r}_{|K} \leq g + \widehat{r}_{|K},
\]
and therefore $h\leq\widehat{r}_{|K}$.
Similarly, we deduce that $h\leq\widehat{s}_{|K}$.
Applying \autoref{prp:fctlInf}, we obtain $z\in S$ satisfying
\[
h\leq\widehat{z}_{|K}, \andSep z\leq r,s.
\]
Choose $z'\in S$ with $z'\ll z$ and $h-\varepsilon \leq \widehat{z'}_{|K}$.
Apply \axiomO{5} to obtain $c\in S$ such that
\[
z'+c\leq nu\leq z+c.
\]
Then
\[
a'+r \leq nu \ll nu \leq z+c \leq r+c.
\]
Since $S$ is weakly cancellative, we deduce that $a'\ll c$.
Similarly, $b'\ll c$.
Then
\[
h-\varepsilon + \widehat{c}_{|K}
\leq \widehat{z'}_{|K} + \widehat{c}_{|K}
\leq \widehat{nu}_{|K}
= g+h,
\]
and thus $\widehat{c}_{|K}\leq g+\varepsilon \ll f$.
Choose $c'\ll c$ such that $a',b'\leq c'$.
Then $c'$ and $c$ have the desired properties, which shows that $L_{f}'$ is upward directed.
\end{proof}

By \autoref{prp:ctblBasedDCPO}, every upward directed set in a countably based \CuSgp{} has a supremum.
This justifies the following:

\begin{dfn}
\label{dfn:alpha}
We define $\alpha\colon\LAff(K)_{++}\to S$ by
\[
\alpha(f) := \sup L_f',
\]
for $f\in\LAff(K)_{++}$.
\end{dfn}

\begin{lma}
\label{prp:wayBelowSoft}
Let $a,b\in S$ satisfy $a\ll b$, and assume that $b$ is soft.
Then $\widehat{a}_{|K}\ll\widehat{b}_{|K}$ in $\LAff(K)$.
\end{lma}
\begin{proof}
Choose $c\in S$ with $a\ll c \ll b$.
Since $b$ is soft, we have $c<_s b$.
This implies that we can choose $\varepsilon>0$ such that $(1+\varepsilon)\widehat{c}\leq\widehat{b}$.
By \cite[Lemma~2.2.5]{Rob13Cone},
\[
\widehat{a} \ll (1+\varepsilon)\widehat{c} \leq \widehat{b}.
\]
Since $S$ is simple, the map $\LAff(F(S))\to\LAff(K)$, given by restricting a function to $K$, is an order-isomorphism.
It follows that $\widehat{a}_{|K}\ll \widehat{b}_{|K}$ in $\LAff(K)$.
\end{proof}

\begin{prp}
\label{prp:maxSoft}
Let $f\in\LAff(K)_{++}$.
Consider the sets
\[
L_f := \big\{ a\in S :  \widehat{a}_{|K}\ll f \text{ in } \LAff(K) \big\},\andSep
S_f = \big\{ a\in S_\txtSoft :  \widehat{a}_{|K}\leq f \big\}.
\]
Then $\alpha(f)=\sup L_f = \max S_f$.
In particular, $\alpha(f)$ is soft.
\end{prp}
\begin{proof}
To verify that $\alpha(f)$ is an upper bound for $L_{f}$, let $a\in L_{f}$.
For every $a'\in S$ with $a'\ll a$ we have $a'\in L_{f}'$ and therefore $a'\leq\sup L_{f}'=\alpha(f)$.
Using $a=\sup\{a':a'\ll a\}$, we deduce that $a\leq\alpha(f)$.
Since $L_{f}'\subseteq L_{f}$, we deduce that $\sup L_{f}$ exists and agrees with $\sup L_{f}'$.

To verify that $\alpha(f)$ is an upper bound for $S_f$, let $a\in S_f$.
For every $a'\in S$ with $a'\ll a$ we have $\widehat{a'}_{|K}\ll\widehat{a}_{|K}$ by \autoref{prp:wayBelowSoft}, which implies that $a'\in L_f$ and thus $a'\leq\alpha(f)$.
Hence, $a\leq\alpha(f)$.

Next, we verify that every element in $L_f$ is below an element in $S_f$.
Let $a\in L_f$.
Then $\widehat{a}_{|K}\ll f$.
Applying \autoref{prp:llInLAff}, we obtain $\varepsilon>0$ such that $\widehat{a}+\varepsilon\leq f$.
By \autoref{prp:structureSimpleCu} and \autoref{prp:realizeRankBySoft}, we can choose $x\in S_\txtSoft$ nonzero with $\widehat{x}_{|K}\leq \varepsilon$.
Since $S_\txtSoft^\times$ is absorbing, it follows that $a+x$ is soft;
see \autoref{pgr:soft}.
Further, $\widehat{a+x}_{|K}\leq f$.
Thus, $a\leq a+x\in S_f$. 
It follows that $\sup S_f=\sup L_f=\alpha(f)$.

Let $B$ be a countable basis for $S$.
To show that $B\cap L_f'$ is upward directed, let $b_1,b_2\in B\cap L_f'$.
Use that $L_f'$ is upward directed to obtain $a'\in L_f'$ with $b_1,b_2\leq a'$.
By definition of $L_f'$, there is $a\in S$ with $a'\ll a$ and $\widehat{a}_{|K} \ll f$.
Choose $b\in B$ with $a'\ll b\ll a$.
Then $b$ belongs to $B\cap L_f'$ and satisfies $a_1,a_2\leq b$, as desired.

We can therefore choose an increasing sequence $(a_n)_n$ in $L_f'$ with $\alpha(f)=\sup_n a_n$.
Since every $a_n$ satisfies $\widehat{a_n}_{|K}\leq f$, we have
\[
\widehat{\alpha(f)}_{|K} = \sup_{n} \widehat{a_n}_{|K} \leq f.
\]
It follows from \autoref{prp:structureSimpleCu} that $L_f'\neq\{0\}$, and thus $\alpha(f)\neq 0$.
Hence, $\alpha(f)$ is either compact or soft.
But if $\alpha(f)$ were compact, then there would be $n$ such that $\alpha(f)\leq a_n$, and therefore  $\widehat{\alpha(f)}_{|K}\ll f$.
Arguing as above, we would obtain $x\in S_\txtSoft$ nonzero such that $\alpha(f)+x\in S_f$, and hence $\alpha(f)+x\leq \alpha(f)$, contradicting that $S$ is stably finite.
It follows that $\alpha(f)$ is soft.
Thus, $\alpha(f)$ belongs to $S_f$ and is therefore the maximal element in $S_f$.
\end{proof}

\begin{thm}
\label{prp:realizeRank}
We have $\widehat{\alpha(f)}=f$, for every $f\in\LAff(K)_{++}$.
In particular, for every $f\in\LAff(K)_{++}$, there exists $a\in S$ with $\widehat{a}_{|K}=f$.
\end{thm}
\begin{proof}
Let $f\in\LAff(K)_{++}$.
As observed in the proof of \autoref{prp:maxSoft}, we have $\widehat{\alpha(f)}\leq f$.
To show the converse inequality, choose an increasing sequence $(g_n)_n$ in $\Aff(K)_{++}$ with supremum $f$.
Choose a decreasing sequence $(\varepsilon_n)_n$ that converges to zero.
We may assume that $g_0-\varepsilon_0$ is strictly positive.
We have $\sup_n (g_n-\varepsilon_n) = f$.

For each $n$, apply \autoref{prp:realizeRankApprox} to obtain $a_n$ satisfying $g_n-\varepsilon_n\leq\widehat{a_n}_{|K}\leq g_n-\tfrac{\varepsilon_n}{2}$.
Using \autoref{prp:llInLAff} at the second step, we have
\[
\widehat{a_n}_{|K}\leq g_n-\tfrac{\varepsilon_n}{2} \ll f.
\]
Thus, $a_n$ belongs to $L_f$, as defined in \autoref{prp:maxSoft}, which implies $a_n\leq\sup L_f=\alpha(f)$.
It follows that
\[
f = \sup_n (g_n-\varepsilon_n)
\leq \sup_n \widehat{a_n}_{|K}
\leq \widehat{\alpha(f)}_{|K},
\]
as desired.
\end{proof}

\begin{pgr}
\label{pgr:GaloisAlpha}
Consider the map $\alpha\colon\LAff(K)_{++}\to S$ as defined in \autoref{dfn:alpha}.
By \autoref{prp:maxSoft}, the image of $\alpha$ is contained in $S_\txtSoft$.
Set $S_\txtSoft^\times:= S_\txtSoft\setminus\{0\}$ and define $\kappa\colon S_\txtSoft^\times\to \LAff(K)_{++}$ by $\kappa(a):=\widehat{a}_{|K}$.
Using \autoref{prp:realizeRank} and \autoref{prp:maxSoft} it follows that $\kappa\circ\alpha=\id$ and $\alpha\circ\kappa\leq\id$.
Given $a\in S_\txtSoft^\times$ and $f\in\LAff(K)_{++}$, we have $a\leq\alpha(f)$ if and only if $\kappa(a)\leq f$.
Thus, $\alpha$ and $\kappa$ form a (order theoretic) Galois connection between $\LAff(K)_{++}$ and $S_\txtSoft^\times$;
see \cite[Definition~O-3.1, p.22]{GieHof+03Domains}.
The map $\alpha$ is the upper adjoint, and the map $\kappa$ is the lower adjoint.
\end{pgr}

\begin{prp}
\label{prp:alpha}
The map $\alpha\colon\LAff(K)_{++}\to S_\txtSoft$ preserves order, infima and suprema of directed sequences.
\end{prp}
\begin{proof}
It follows from \autoref{prp:maxSoft} that $\alpha$ is order-preserving.
By \cite[Theorem~O-3.3, p.24]{GieHof+03Domains}, the upper adjoint of a Galois connection preserves arbitrary existing infima.
Since $K$ is a Choquet simplex, $\LAff(K)_{++}$ is an inf-semilattice;
see \autoref{prp:LAffRiesz}.
Thus, given $f,g\in\LAff(K)_{++}$, we have
\[
\alpha(f\wedge g)=\alpha(f)\wedge\alpha(g).
\]
By \cite[Theorem~IV-1.4, p.268]{GieHof+03Domains}, if the lower adjoint of a Galois connection preserves the way-below relation, and if the source of the lower adjoint is a domain, then the upper adjoint preserves suprema of increasing nets.
By \autoref{prp:wayBelowSoft}, the map $\kappa$ preserves the way-below relation, whence $\alpha$ preserves suprema of increasing sequences.
\end{proof}

\begin{qst}
Is the map $\alpha\colon\LAff(F(S))_{++}\to S_\txtSoft$ additive?
Does it preserve the way-below relation?
\end{qst}

\begin{thm}
\label{prp:realizeRankCa}
Let $A$ be a separable, unital, simple, non-elementary \ca{} with stable rank one.
Set $K:=\QT_{1\mapsto 1}(A)$.
Then for every $f\in\LAff(K)_{++}$ there exists $a\in(A\otimes\KK)_+$ such that $d_\tau(a)=f(\tau)$ for every $\tau\in K$.
\end{thm}
\begin{proof}
Let $f\in\LAff(K)_{++}$.
Set $S:=\Cu(A)$.
It follows from the properties of $A$ that $S$ is a countably based, simple, non-elementary \CuSgp{} satisfying \axiomO{5}.
By \cite[Theorem~4.3]{RorWin10ZRevisited}, the Cuntz semigroup of a \ca{} with stable rank one is weakly cancellative;
see also \cite[Chapter~4]{AntPerThi18:TensorProdCu}.
It follows from \autoref{prp:RieszSR1} that $S$ satisfies \axiomO{6+}.

The class $u:=[1_A]$ is a compact, full element in $S$.
Under the identification of $\QT(A)$ with $F(S)$, see \autoref{pgr:qt}, the set $K$ corresponds to $F_{u\mapsto 1}(S)$.
It follows from \autoref{prp:EdwardsCa} that $S$ satisfies Edwards' condition for $\partial_e K$.

We may therefore apply \autoref{prp:realizeRank} for $S$ and $f$ to obtain $s\in S$ such that $\widehat{s}_{|K}=f$.
Then any $a\in(A\otimes\KK)_+$ with $s=[a]$ has the desired property.
\end{proof}

\begin{cor}
\label{prp:almDivFromAlmUnperf}
Let $A$ be a separable, unital, simple, non-elementary \ca{} with stable rank one.
Assume that $\Cu(A)$ is almost unperforated, that is, that $A$ has strict comparison of positive elements.
Then $\Cu(A)$ is almost divisible.
Moreover, there are order-isomorphisms
\[
\Cu(A)
\cong V(A) \sqcup \LAff(\QT_{1\mapsto 1}(A))_{++}
\cong\Cu(A\otimes\mathcal{Z}).
\]
The map $A\to A\otimes\mathcal{Z}$, $a\mapsto a\otimes 1$, induces the isomorphism $\Cu(A)\cong\Cu(A\otimes\mathcal{Z})$.
\end{cor}
\begin{proof}
Set $K:=\QT_{1\mapsto 1}(A)$.
It follows from the properties of $A$ that $\Cu(A)$ is a countably based, simple, stably finite, non-elementary \CuSgp{} satisfying \axiomO{5} and \axiomO{6}.
By \autoref{prp:realizeRankCa}, $\Cu(A)$ satisfies statement~(4) of \autoref{prp:charAlmDiv}.
We deduce that $\Cu(A)$ is almost divisible and that $\kappa\colon \Cu(A)_\txtSoft\to\LAff(K)_{++}$ is an order-isomorphism.
There is a natural affine homeomorphism between the simplex of normalized $2$-quasitraces on $A$ and $A\otimes\mathcal{Z}$.
We therefore have an order-isomorphisms
\[
\Cu(A)_\txtSoft \cong \LAff(K)_{++} \cong \Cu(A\otimes\mathcal{Z})_\txtSoft.
\]
We have an order-isomorphism between $\Cu(A)_c$ and the Muray-von Neumann semigroup $V(A)$.
Since $A$ is unital and has stable rank one, $V(A)$ is order-isomorphic to $K_0(A)_+$, the positive part of the partially ordered group $K_0(A)$.
Similarly, $V(A\otimes\mathcal{Z})\cong K_0(A\otimes\mathcal{Z})_+$.

By \cite[Theorem~1]{GonJiaSu00ObstrZstability}, the map $\iota$ induces an order-isomorphism $K_0(A)\to K_0(A\otimes\mathcal{Z})$ if and only if $K_0(A)$ is weakly unperforated.
This condition is verified using \cite[Section~3]{Ror04StableRealRankZ}.
Thus, we have order-isomorphisms
\[
\Cu(A)_c \cong V(A) \cong K_0(A)_+ \cong K_0(A\otimes\mathcal{Z})_+
\cong V(A\otimes\mathcal{Z})
\cong \Cu(A\otimes\mathcal{Z})_c.
\]
The result follows using the decomposition of a simple \CuSgp{} into its compact and soft part;
see \autoref{prp:structureSimpleCu}(1).
\end{proof}

\section{The Toms-Winter conjecture}
\label{sec:TW}

\begin{cnj}[Toms-Winter]
\label{cnj:TW}
Let $A$ be a separable, unital, simple, non-el\-e\-men\-ta\-ry, nuclear \ca{}.
Then the following are equivalent:
\begin{enumerate}
\item
$A$ has finite nuclear dimension.
\item
$A$ is $\mathcal{Z}$-stable, that is, $A\cong A\otimes\mathcal{Z}$.
\item
$\Cu(A)$ is almost unperforated. 
\end{enumerate}
\end{cnj}

\begin{rmks}
\label{rmk:TW}
(1)
The nuclear dimension is a non-commutative analogue of topological covering dimension, introduced by Winter-Zacharias \cite[Definition~2.1]{WinZac10NuclDim}.

(2)
The Jiang-Su algebra $\mathcal{Z}$, introduced in \cite{JiaSu99Projectionless}, is a separable, unital, simple, non-el\-e\-men\-ta\-ry, approximately subhomogeneous \ca{} with unique tracial state and $K_0(\mathcal{Z})\cong\ZZ$ and $K_1(\mathcal{Z})=0$.
Being $\mathcal{Z}$-stable is considered as the \ca{ic} analogue of being a McDuff von Neumann factor.

(3)
Let $A$ be a unital, simple \ca{}.
Recall that $A$ is said to have \emph{strict comparison of positive elements} if for all $a,b\in(A\otimes\KK)_+$ we have $a\precsim b$ whenever $d_\tau(a)<d_\tau(b)$ for all $\tau\in\QT_{1\mapsto 1}(A)$.
Let $S$ be a simple \CuSgp{} with a compact, full element $u\in S$.
Then $S$ is almost unperforated if and only if for all $a,b\in S$ we have $a\leq b$ whenever $\lambda(a)<\lambda(b)$ for all $\lambda\in F_{u\mapsto 1}(S)$;
see \cite[Proposition~5.2.14]{AntPerThi18:TensorProdCu}.
For $S=\Cu(A)$ we have a natural identification of $\QT_{1\mapsto 1}(A)$ with $F_{[1]\mapsto 1}(\Cu(A))$, that maps $\tau$ to $d_\tau$.
It follows that $A$ has strict comparison of positive elements if and only if $\Cu(A)$ is almost unperforated.

(4)
The regularity conjecture of Toms-Winter is intimately connected to the Elliott classification program, which seeks to classify simple, nuclear \ca{s} by $K$-theoretical and tracial data.
Examples of Toms, \cite{Tom08ClassificationNuclear}, showed that an additional regularity assumption is necessary to obtain such a classification.

Due to a series of remarkable breakthroughs in the last three years, building on an extensive body of work over decades by numerous people, the classification of separable, unital, simple, non-el\-e\-men\-ta\-ry \ca{s} with finite nuclear dimension that satisfy the universal coefficient theorem (UCT) has been completed;
see \cite{EllGonLinNiu15arX:classFinDR2} and \cite[Corollary~D]{TikWhiWin17QDNuclear}.

In particular, it follows from \autoref{prp:TW-ASHsr1} that separable, unital, simple, approximately subhomogeneous \ca{s} with stable rank one and strict comparison of positive elements are classified by $K$-theoretic and tracial data.
\end{rmks}

\begin{pgr}
\label{pgr:TW}
The implications `(1)$\Rightarrow$(2)' and `(2)$\Rightarrow$(3)' of the Toms-Winter conjecture have been confirmed in general.
The first implication is due to Winter, \cite[Corollary~7.3]{Win12NuclDimZstable}.
The second was shown by R{\o}rdam, \cite[Theorem~4.5]{Ror04StableRealRankZ}.

The implication `(2)$\Rightarrow$(1)' was shown to hold under the additional assumption that $\partial_e T(A)$ is compact,
see \cite[Theorem~B]{BBSTWW19},
generalizing the earlier solution of the monotracial case,
\cite{MatSat14DecRankUHFAbs}, \cite{SatWhiWin15NuclDimZstab}.
Very recently, it was shown that the implication `(2)$\Rightarrow$(1)' holds in general;
see \autoref{rmk:CETWW}.

The implication `(3)$\Rightarrow$(2)' was shown to hold under the additional assumption that $\partial_e T(A)$ is compact and finite dimensional by independent works of Kirchberg-R{\o}rdam, \cite{KirRor14CentralSeq}, Sato, \cite{Sat12arx:TraceSpace}, and Toms-White-Winter, \cite{TomWhiWin15ZStableFdBauer}, extending the work of Matuia-Sato, \cite{MatSat12StrComparison}, that covered the case of finitely many extreme traces.

This was further extended by Zhang, \cite{Zha14TracialStateNoncpctBdry}, who showed that the implication `(3)$\Rightarrow$(2)' also holds in certain cases where $\partial_e T(A)$ is finite dimensional but not necessarily compact.
\end{pgr}

\begin{pgr}
Let $A$ be a separable, unital, simple, non-elementary, nuclear \ca{}.
The Cuntz semigroup of every $\mathcal{Z}$-stable \ca{} is almost unperforated and almost divisible.
Thus, the Toms-Winter conjecture predicts in particular that $\Cu(A)$ is almost divisible whenever it is almost unperforated.
Moreover, to prove the implication `(3)$\Rightarrow$(2)' of the Toms-Winter conjecture, the verification of almost divisibility of $\Cu(A)$ is crucial.
Winter showed that it is even enough under the additional assumption that $A$ has locally finite nuclear dimension;
see \cite[Corollary~7.4]{Win12NuclDimZstable}.
By \autoref{prp:almDivFromAlmUnperf}, if $A$ has stable rank one, then $\Cu(A)$ is almost divisible whenever it is almost unperforated.
Combined with the result of Winter, we obtain:
\end{pgr}

\begin{thm}
\label{prp:TWlocFinNuclDim}
Let $A$ be a separable, unital, simple, non-elementary \ca{} with stable rank one and locally finite nuclear dimension.
Then $A$ is $\mathcal{Z}$-stable if and only if $A$ has strict comparison of positive elements.
(See also \autoref{rmk:CETWW}.)
\end{thm}

It was shown by Elliott-Niu-Santiago-Tikuisis, \cite[Theorem~1.1]{EllNiuSanTik15arX:drASH}, that the nuclear dimension of every (not necessarily simple or unital) separable, approximately subhomogeneous, $\mathcal{Z}$-stable \ca{} is finite.
Ng-Winter, \cite{NgWin06NoteSH}, proved that every separable, approximately subhomogeneous \ca{} has locally finite nuclear dimension.
Combining with the above result, we deduce:

\begin{thm}
\label{prp:TW-ASHsr1}
The Toms-Winter conjecture holds for approximately subhomogeneous \ca{s} with stable rank one.
(See also \autoref{rmk:CETWW}.)
\end{thm}

\begin{rmk}
The class of algebras covered by \autoref{prp:TW-ASHsr1} is very rich and it includes both $\mathcal{Z}$-stable and non-$\mathcal{Z}$-stable algebras.
For example, by \cite[Theorem~4.1]{EllHoTom09ClassSimpleSR1}, all diagonal AH-algebras have stable rank one, whence they are covered by \autoref{prp:TW-ASHsr1}.
This includes in particular the Villadsen algebras of first type, \cite{Vil98SimpleCaPerforation}, and Toms' celebrated examples, \cite{Tom08ClassificationNuclear}.
\end{rmk}

\begin{pgr}
\label{pgr:crossedProd}
Let $X$ be an infinite, compact, metrizable space, and let $\alpha\colon X\to X$ be a minimal homeomorphism.
Then $\alpha$ induces an automorphism of $C(X)$.
The crossed product $A:=C(X)\rtimes\ZZ$ is a separable, unital, simple, non-elementary, nuclear \ca{}.
Let $u\in A$ be the canonical unitary implementing $\alpha$.
Choose $x\in X$ and set
\[
A_x :=  C^*\big( C(X), C_0(X\setminus\{x\})u \big) \subseteq A,
\]
the `orbit breaking subalgebra' of $A$ corresponding to $x$.
Then $A_x$ is a separable, unital, simple, non-elementary, approximately subhomogeneous \ca{}.
By combining \cite[Theorem~7.10]{Phi14arX:LargeSub} and \cite[Theorem~4.6]{ArcPhi15arX:SRCentrLargeSub}, we obtain that $A_x$ is a `centrally large subalgebra' of $A$, in the sense of \cite[Definition~3.2]{ArcPhi15arX:SRCentrLargeSub}.
By \cite[Theorem~6.14]{Phi14arX:LargeSub}, $\Cu(A)$ is almost unperforated if and only if $\Cu(A_x)$ is.
By Theorem~3.3 and Corollary~3.5 in \cite{ArcBucPhi18:CentrLargeSub}, $A$ is $\mathcal{Z}$-stable if and only if $A_x$ is.

Since $A_x$ is approximately subhomogeneous, it follows from \cite[Theorem~A]{EllNiuSanTik15arX:drASH} that $A_x$ has finite nuclear dimension (denoted $\dimnuc(A_x)<\infty$) whenever $A_x$ is $\mathcal{Z}$-stable.
By \cite[Corollary~4.8]{EllNiuSanTik15arX:drASH}, the analogous statement holds also for $A$.
We show the various implications in the following diagram.
The statements in the three columns correspond to the three statements of the Toms-Winter conjecture.
The horizontal implications towards the right hold in general;
see \autoref{pgr:TW}.
\[
\xymatrix{
\dimnuc(A)<\infty \ar[r]
& A\cong A\otimes\mathcal{Z} \ar[r] \ar@/_1pc/[l]_{\text{ENST}} \ar@{<->}[d]^{\text{ABP}}
& \Cu(A) \text{ almost unperforated} \ar@{<->}[d]^{\text{P}} \\
\dimnuc(A_x)<\infty \ar[r]
& A_x\cong A_x\otimes\mathcal{Z} \ar[r] \ar@/_1pc/[l]_{\text{ENST}}
& \Cu(A_x) \text{ almost unperforated}
}\,
\]
Thus, to verify the Toms-Winter conjecture for $A$ and $A_x$ it remains only to verify that almost unperforation of their Cuntz semigroups implies $\mathcal{Z}$-stability.
Moreover, $A$ satisfies the Toms-Winter conjecture if and only if $A_x$ does.

The large subalgebra $A_x$ is approximately subhomogeneous (and separable, unital, simple, non-el\-e\-men\-ta\-ry, nuclear).
Therefore, we may apply \autoref{prp:TW-ASHsr1} to verify the Toms-Winter conjecture for $A_x$ if we know that $A_x$ has stable rank one.
It is conjectured by Archey-Niu-Phillips, \cite[Conjecture~7.2]{ArcPhi15arX:SRCentrLargeSub}, that $A$ always has stable rank one.
By \cite[Theorem~6.3]{ArcPhi15arX:SRCentrLargeSub}, $A$ has stable rank one whenever $A_x$ does and they use this to verify their conjecture in the case that the dynamical system has a Cantor factor.
We obtain:
\end{pgr}

\begin{thm}
\label{prp:TWcrossedProducts}
Let $X$ be an infinite, compact, metrizable space, together with a minimal homeomorphism on $X$.
Then the Toms-Winter conjecture holds for the crossed product $A=C(X)\rtimes\ZZ$ if it has large subalgebra with stable rank one.
In particular, the Toms-Winter conjecture holds for the crossed product $A=C(X)\rtimes\ZZ$ if the dynamical system has a Cantor factor.
\end{thm}

It was also shown by R{\o}rdam that every unital, simple, stably finite, $\mathcal{Z}$-stable \ca{} has stable rank one;
see \cite[Theorem~6.7]{Ror04StableRealRankZ}.
Therefore, if a stably finite \ca{} satisfies either condition~(1) or~(2) of the Toms-Winter conjecture, then it has stable rank one.
We are led to ask the following:

\begin{qst}
Let $A$ be a separable, unital, simple, stably finite, nuclear \ca{} such that $\Cu(A)$ is almost unperforated.
Does $A$ have stable rank one?
\end{qst}

\begin{rmk}
\label{rmk:CETWW}
After this paper had been submitted, it was shown that the implication `(2)$\Rightarrow$(1)' of the Toms-Winter conjecture holds in general;
see \cite[Theorem~A]{CasEviTikWhiWin19arX:NucDimSimple}.
It follows that Theorems~\ref{prp:TWlocFinNuclDim} and~\ref{prp:TW-ASHsr1} can be generalized as follows:
The Toms-Winter conjecture holds for \ca{s} with locally finite nuclear dimension and stable rank one.
\end{rmk}


\providecommand{\etalchar}[1]{$^{#1}$}
\providecommand{\bysame}{\leavevmode\hbox to3em{\hrulefill}\thinspace}
\providecommand{\noopsort}[1]{}
\providecommand{\mr}[1]{\href{http://www.ams.org/mathscinet-getitem?mr=#1}{MR~#1}}
\providecommand{\zbl}[1]{\href{http://www.zentralblatt-math.org/zmath/en/search/?q=an:#1}{Zbl~#1}}
\providecommand{\jfm}[1]{\href{http://www.emis.de/cgi-bin/JFM-item?#1}{JFM~#1}}
\providecommand{\arxiv}[1]{\href{http://www.arxiv.org/abs/#1}{arXiv~#1}}
\providecommand{\doi}[1]{\url{http://dx.doi.org/#1}}
\providecommand{\MR}{\relax\ifhmode\unskip\space\fi MR }
\providecommand{\MRhref}[2]{%
	\href{http://www.ams.org/mathscinet-getitem?mr=#1}{#2}
}
\providecommand{\href}[2]{#2}

\end{document}